\documentclass[11pt, noadjust]{amsart}
\usepackage[dvipsnames]{xcolor}
\usepackage[utf8]{inputenc}
\usepackage{amssymb,amsmath,amsthm,amsfonts,fullpage,float,latexsym,bbm,microtype,cite,fancyvrb}
\usepackage{tikz}
\usetikzlibrary{calc}
\usepackage{enumerate}
\usetikzlibrary{decorations.pathreplacing,patterns}
\usepackage{hyperref}
\hypersetup{colorlinks=true, linkcolor=blue, citecolor=magenta, filecolor=magenta, urlcolor=magenta}
\usepackage{graphicx}
\usepackage{ytableau}
\usepackage{subcaption}
\usepackage{makecell}
\usepackage{tablefootnote}
\usepackage{bm}

\definecolor{Midori}{HTML}{2A603B}
\definecolor{Scarlet}{HTML}{CF3A24}
\definecolor{Gamboge}{HTML}{FFB61E}
\definecolor{Lapis}{HTML}{1F4788}

\makeatletter
\newtheorem*{rep@theorem}{\rep@title}\newcommand{\newreptheorem}[2]{%
\newenvironment{rep#1}[1]{%
\def\rep@title{\bf #2 \ref{##1}}%
\begin{rep@theorem}}%
{\end{rep@theorem}}}
\makeatother
\newreptheorem{theorem}{Theorem}
\newtheorem{theorem}{Theorem}[section]
\newtheorem{proposition}[theorem]{Proposition}

\newtheorem{conjecture}[theorem]{Conjecture}
\newtheorem{lemma}[theorem]{Lemma}
\newtheorem{corollary}[theorem]{Corollary}
\theoremstyle{definition}
\newtheorem{remark}[theorem]{Remark}
\newtheorem{definition}[theorem]{Definition}

\newtheorem{example}[theorem]{Example}

\newcommand{\Z}{\mathbb{Z}}
\newcommand{\C}{\mathbb{C}}

\DeclareMathOperator{\Hilb}{Hilb}
\DeclareMathOperator{\Frob}{Frob}
\DeclareMathOperator{\stair}{stair}
\DeclareMathOperator{\Stir}{Stir}
\DeclareMathOperator{\SW}{SW}
\DeclareMathOperator{\SF}{SF}
\DeclareMathOperator{\sminv}{sminv}
\DeclareMathOperator{\sdinv}{sdinv}
\DeclareMathOperator{\Split}{Split}

\DeclareMathOperator{\wt}{wt}
\DeclareMathOperator{\Set}{Set}
\DeclareMathOperator{\Comp}{co}
\DeclareMathOperator{\Asc}{Asc}
\DeclareMathOperator{\SYT}{SYT}
\DeclareMathOperator{\maj}{maj}
\DeclareMathOperator{\dist}{dist}
\DeclareMathOperator{\des}{des}
\DeclareMathOperator{\Char}{char}
\DeclareMathOperator{\MS}{MS}
\DeclareMathOperator{\up}{up}

\DeclareRobustCommand{\qbinom}{\genfrac[]{0pt}{}}
\newcommand{\multiset}[2]{\big(\!\binom{#1}{#2}\!\big)}

\newcommand{\motzkin}[2]{
    \begin{tikzpicture}[scale=0.6, baseline=(current bounding box.center)]
        \coordinate (current) at (0,0);
        \ifcsname c@step\endcsname
        \else
            \newcounter{step}
        \fi
        \setcounter{step}{1}
        \foreach \step in {#1} {
            \ifnum\pdfstrcmp{\step}{up}=0
                \draw[thick] (current) -- ++(1,1);
                \coordinate (current) at ($(current) + (1,1)$);
            \fi
            \ifnum\pdfstrcmp{\step}{flat-theta}=0
                \draw[thick] (current) -- ++(1,0) node[midway, above, font=\footnotesize] {$\theta_{\arabic{step}}$};
                \coordinate (current) at ($(current) + (1,0)$);
            \fi
            \ifnum\pdfstrcmp{\step}{flat-xi}=0
                \draw[thick] (current) -- ++(1,0) node[midway, above, font=\footnotesize] {$\xi_{\arabic{step}}$};
                \coordinate (current) at ($(current) + (1,0)$);
            \fi
            \ifnum\pdfstrcmp{\step}{down}=0
                \draw[thick] (current) -- ++(1,-1) node[midway, above, font=\footnotesize, xshift=2pt, yshift=1pt] {$\theta_{\arabic{step}} \xi_{\arabic{step}}$};
                \coordinate (current) at ($(current) + (1,-1)$);
            \fi
            \stepcounter{step}
        }
        \node[anchor=north] at (0.5*\arabic{step}, -0.5) {\normalsize #2};
    \end{tikzpicture}
}

\newcommand{\motzkinpath}[4]{
    \begin{tikzpicture}[scale=0.6, baseline=(current bounding box.center)]
        \foreach \i [count=\j from 1] in {#2} {
            \ifnum\i>0
                \foreach \k in {0,...,\numexpr\i-1} {
                    \draw (\j-1,\k) rectangle (\j,\k+1);
                }
            \else
                \draw (\j-1,0) -- (\j,0);
            \fi
        }
        \foreach \i [count=\j from 1] in {#3} {
            \ifnum\i>0
                \foreach \k in {0,...,\numexpr\i-1} {
                    \draw[fill=gray!40] (\j-1,\k) rectangle (\j,\k+1);
                }
            \fi
        }
        \coordinate (current) at (0,0);
        \ifcsname c@step\endcsname
        \else
            \newcounter{step}
        \fi
        \setcounter{step}{1}
        \foreach \step in {#1} {
            \ifnum\pdfstrcmp{\step}{up}=0
                \draw[thick] (current) -- ++(1,1);
                \coordinate (current) at ($(current) + (1,1)$);
            \fi
            \ifnum\pdfstrcmp{\step}{flat-theta}=0
                \draw[thick] (current) -- ++(1,0) node[midway, above, font=\footnotesize] {$\theta_{\arabic{step}}$};
                \coordinate (current) at ($(current) + (1,0)$);
            \fi
            \ifnum\pdfstrcmp{\step}{flat-xi}=0
                \draw[thick] (current) -- ++(1,0) node[midway, above, font=\footnotesize] {$\xi_{\arabic{step}}$};
                \coordinate (current) at ($(current) + (1,0)$);
            \fi
            \ifnum\pdfstrcmp{\step}{down}=0
                \draw[thick] (current) -- ++(1,-1) node[midway, above, font=\footnotesize, xshift=2pt, yshift=1pt] {$\theta_{\arabic{step}} \xi_{\arabic{step}}$};
                \coordinate (current) at ($(current) + (1,-1)$);
            \fi
            \ifnum\pdfstrcmp{\step}{skip}=0
                \coordinate (current) at ($(current) + (1,-1)$);
            \fi
            \stepcounter{step}
        }
        \node[anchor=north] at (0.5*\arabic{step}, -0.5) {\normalsize #4};
    \end{tikzpicture}
}

\newcommand{\motzkinpathbig}[7]{
    \begin{tikzpicture}[scale=0.6, baseline=(current bounding box.center)]
        \foreach \i [count=\j from 1] in {#2} {
            \ifnum\i>0
                \foreach \k in {0,...,\numexpr\i-1} {
                    \draw (\j-1,\k) rectangle (\j,\k+1);
                }
            \else
                \draw (\j-1,0) -- (\j,0);
            \fi
        }
        \foreach \i [count=\j from 1] in {#3} {
            \ifnum\i>0
                \foreach \k in {0,...,\numexpr\i-1} {
                    \draw[pattern = crosshatch] (\j-1,\k) rectangle (\j,\k+1);
                }
            \fi
        }
        \coordinate (current) at (0,0);
        \ifcsname c@step\endcsname
        \else
            \newcounter{step}
        \fi
        \setcounter{step}{1}
        \foreach \step in {#1} {
            \ifnum\pdfstrcmp{\step}{up}=0
                \draw[thick] (current) -- ++(1,1);
                \coordinate (current) at ($(current) + (1,1)$);
            \fi
            \ifnum\pdfstrcmp{\step}{flat-theta}=0
                \draw[thick] (current) -- ++(1,0) node[midway, above, font=\footnotesize] {$\theta_{\arabic{step}}$};
                \coordinate (current) at ($(current) + (1,0)$);
            \fi
            \ifnum\pdfstrcmp{\step}{flat-xi}=0
                \draw[thick] (current) -- ++(1,0) node[midway, above, font=\footnotesize] {$\xi_{\arabic{step}}$};
                \coordinate (current) at ($(current) + (1,0)$);
            \fi
            \ifnum\pdfstrcmp{\step}{down}=0
                \draw[thick] (current) -- ++(1,-1) node[midway, above, font=\footnotesize, xshift=7pt, yshift=1pt] {$\theta_{\arabic{step}} \xi_{\arabic{step}}$};
                \coordinate (current) at ($(current) + (1,-1)$);
            \fi
            \ifnum\pdfstrcmp{\step}{skip}=0
                \coordinate (current) at ($(current) + (1,-1)$);
            \fi
            \stepcounter{step}
        }
        \draw[ultra thick, dashed, red] (#5, 0) -- (#5, #7) node[pos=1, above] {$T_1$};
        \draw[ultra thick, dashed, blue] (#6, 0) -- (#6, #7) node[pos=1, above] {$T_2$};
        \draw [decorate,decoration={brace,amplitude=5pt,mirror,raise=1ex}]
        (0,0) -- (#5,0) node[midway,yshift=-2em]{$d+1$} node[midway,yshift=-3.5em]{Region 0};
        \draw [decorate,decoration={brace,amplitude=5pt,mirror,raise=1ex}]
        (#5,0) -- (#6,0) node[midway,yshift=-2em]{$n-d-1-k$} node[midway,yshift=-3.5em]{Region 1};
        \draw [decorate,decoration={brace,amplitude=5pt,mirror,raise=1ex}]
        (#6,0) -- (14,0) node[midway,yshift=-2em]{$k$} node[midway,yshift=-3.5em]{Region 2};
        \node[anchor=north] at (0.5*\arabic{step}, -0.5) {\normalsize #4};
    \end{tikzpicture}
}

\newcommand{\motzkinpathmiddle}[7]{
    \begin{tikzpicture}[scale=0.55, baseline=(current bounding box.center)]
        \foreach \i [count=\j from 1] in {#2} {
            \ifnum\numexpr\j-1<#6
                \ifnum\i>0
                    \foreach \k in {0,...,\numexpr\i-1} {
                        \draw (\j-1,\k) rectangle (\j,\k+1);
                    }
                \fi
            \fi
        }
        \foreach \i [count=\j from 1] in {#3} {
            \ifnum\numexpr\j-1<#6
                \ifnum\i>0
                    \foreach \k in {0,...,\numexpr\i-1} {
                        \draw[pattern = crosshatch] (\j-1,\k) rectangle (\j,\k+1);
                    }
                \fi
            \fi
        }
        \coordinate (current) at (0,0);
        \ifcsname c@step\endcsname
        \else
            \newcounter{step}
        \fi
        \setcounter{step}{1}
        \foreach \step in {#1} {
            \ifnum\pdfstrcmp{\step}{up}=0
                \ifnum\value{step} > #5 %NEW
                    \draw[thick] (current) -- ++(1,1);
                    \coordinate (current) at ($(current) + (1,1)$);
                \else
                    \coordinate (current) at ($(current) + (1,1)$);
                \fi %NEW
            \fi
            \ifnum\pdfstrcmp{\step}{flat-xi}=0
                \draw[thick] (current) -- ++(1,0) node[midway, above, font=\footnotesize] {$\xi_{\arabic{step}}$};
                \coordinate (current) at ($(current) + (1,0)$);
            \fi
            \ifnum\pdfstrcmp{\step}{skip}=0
                \coordinate (current) at ($(current) + (1,-1)$);
            \fi
            \stepcounter{step}
        }
        \node[anchor=north] at (0.5*\arabic{step}, -0.5) {\normalsize #4};
    \end{tikzpicture}
}

\newcommand{\motzkinpathnolabel}[3]{
    \begin{tikzpicture}[scale=0.6, baseline=(current bounding box.center)]
        \foreach \i [count=\j from 1] in {#2} {
            \ifnum\i>0
                \foreach \k in {0,...,\numexpr\i-1} {
                    \draw (\j-1,\k) rectangle (\j,\k+1);
                }
            \else
                \draw (\j-1,0) -- (\j,0);
            \fi
        }
        \foreach \i [count=\j from 1] in {#3} {
            \ifnum\i>0
                \foreach \k in {0,...,\numexpr\i-1} {
                    \draw[fill=gray!40] (\j-1,\k) rectangle (\j,\k+1);
                }
            \fi
        }
        \coordinate (current) at (0,0);
        \ifcsname c@step\endcsname
        \else
            \newcounter{step}
        \fi
        \setcounter{step}{1}
        \foreach \step in {#1} {
            \ifnum\pdfstrcmp{\step}{up}=0
                \draw[thick] (current) -- ++(1,1);
                \coordinate (current) at ($(current) + (1,1)$);
            \fi
            \ifnum\pdfstrcmp{\step}{flat-theta}=0
                \draw[thick] (current) -- ++(1,0) node[midway, above, font=\footnotesize] {$\theta_{\arabic{step}}$};
                \coordinate (current) at ($(current) + (1,0)$);
            \fi
            \ifnum\pdfstrcmp{\step}{flat-xi}=0
                \draw[thick] (current) -- ++(1,0) node[midway, above, font=\footnotesize] {$\xi_{\arabic{step}}$};
                \coordinate (current) at ($(current) + (1,0)$);
            \fi
            \ifnum\pdfstrcmp{\step}{down}=0
                \draw[thick] (current) -- ++(1,-1) node[midway, above, font=\footnotesize, xshift=2pt, yshift=1pt] {$\theta_{\arabic{step}} \xi_{\arabic{step}}$};
                \coordinate (current) at ($(current) + (1,-1)$);
            \fi
            \ifnum\pdfstrcmp{\step}{skip}=0
                \coordinate (current) at ($(current) + (1,-1)$);
            \fi
            \stepcounter{step}
        }
    \end{tikzpicture}
}
\begin{document}

\title[The $(1,2)$-bosonic-fermionic coinvariant ring]{A conjectural basis for the $(1,2)$-bosonic-fermionic coinvariant ring}

\author[J. Lentfer]{John Lentfer}
\address{Department of Mathematics\\
         University of California, Berkeley, CA, USA}
\email{jlentfer@berkeley.edu}

%%%%%%%%%%%%%%%%%%%%%%%%%%%%%%%%%%%%%%%%%%%%%%%%%%%%%%%%%%%%%%%%%%%%%

\begin{abstract}
We give the first conjectural construction of a monomial basis for the coinvariant ring $R_n^{(1,2)}$, for the symmetric group $\mathfrak{S}_n$ acting on one set of bosonic (commuting) and two sets of fermionic (anticommuting) variables. 
Our construction interpolates between the modified Motzkin path basis for $R_n^{(0,2)}$ of Kim--Rhoades (2022) and the super-Artin basis for $R_n^{(1,1)}$ conjectured by Sagan--Swanson (2024) and proven by Angarone et al. (2025). 
We prove that our proposed basis has cardinality $2^{n-1}n!$, aligning with a conjecture of Zabrocki (2020) on the dimension of $R_n^{(1,2)}$, and show how it gives a combinatorial expression for the Hilbert series. 
We also conjecture a Frobenius series for $R_n^{(1,2)}$.
We show that these proposed Hilbert and Frobenius series are equivalent to conjectures of Iraci, Nadeau, and Vanden Wyngaerd (2024) on $R_n^{(1,2)}$ in terms of segmented Smirnov words, by exhibiting a weight-preserving bijection between our proposed basis and their segmented permutations.
We extend some of their results on the sign character to hook characters, and give a formula for the $m_\mu$ coefficients of the conjectural Frobenius series.
Finally, we conjecture a monomial basis for the analogous ring in type $B_n$, and show that it has cardinality $4^nn!$.
\end{abstract}

\maketitle
%%%%%%%%%%%%%%%%%%%%%%%%%%%%%%%%%%%%%%%%%%%%%%%%%%%%%%%%%%%%%%%%%%%%%
%\tableofcontents

\section{Introduction}\label{sec:introduction}
The classical coinvariant ring $R_n^{(1,0)} = \C[\bm{x}_n]/\allowbreak \langle \C[\bm{x}_n]_+^{\mathfrak{S}_n} \rangle$ is the quotient ring of a polynomial ring in $n$ variables $\bm{x}_n = \{x_1,\ldots,x_n\}$ by $\mathfrak{S}_n$-invariant polynomials with no constant term.
It is well-known (see for example \cite[Section 1.5]{Haiman1994}) to have dimension $n!$, Hilbert series $[n]_q!$, and Frobenius series \[\sum_{\lambda \vdash n} \sum_{T \in \SYT(\lambda)} q^{\maj(T)}s_\lambda.\]
An important basis of the classical coinvariant ring is the Artin basis $\{x_1^{\alpha_1}x_2^{\alpha_2}\cdots x_n^{\alpha_n}\, | \,0\leq \alpha_i \leq i-1\}$.

In 1994, Haiman \cite{Haiman1994} introduced the diagonal coinvariant ring $R_n^{(2,0)} = \C[\bm{x}_n,\bm{y}_n]/\allowbreak \langle \C[\bm{x}_n,\bm{y}_n]_+^{\mathfrak{S}_n} \rangle$, which extends the classical coinvariant ring to two sets of $n$ variables. 
Here and throughout, $\mathfrak{S}_n$ acts diagonally by permuting the indices of the variables. 
Given a polynomial ring $R$, the notation $\langle R_+^{\mathfrak{S}_n} \rangle$ denotes the ideal generated by all polynomials in $R$ which are invariant under the diagonal action of $\mathfrak{S}_n$, with no constant term.
In 2002, Haiman \cite{Haiman2002} proved that $R_n^{(2,0)}$ has dimension $(n+1)^{n-1}$, Hilbert series $\langle \nabla e_n , h_1^n\rangle$, and Frobenius series $\nabla e_n$, using several deep results in algebraic geometry.
A combinatorial formula for its Hilbert series was conjectured by Haglund and Loehr \cite{HaglundLoehr2005}, and eventually was proven as a consequence of the more general shuffle theorem, conjectured by Haglund, Haiman, Loehr, Remmel, and Ulyanov \cite{HHLRU2005} and proven by Carlsson and Mellit \cite{CarlssonMellit2018}, which gives a combinatorial formula for $\nabla e_n$.
A monomial basis was given by Carlsson and Oblomkov \cite{CarlssonOblomkov2018}.

Recently, there has been interest (see \cite{BergeronOPAC, Bergeron2020, BHIR, Zabrocki2020}) in extending the setting to include coinvariant rings with $k$ sets of $n$ commuting variables $\bm{x}_n, \bm{y}_n, \bm{z}_n, \ldots$ and $j$ sets of $n$ anticommuting variables $\bm{\theta}_n, \bm{\xi}_n, \bm{\rho}_n, \ldots$.
We denote this by 
\begin{equation} R_n^{(k,j)} = \C[\underbrace{\bm{x}_n, \bm{y}_n, \bm{z}_n, \ldots}_k, \underbrace{\bm{\theta}_n, \bm{\xi}_n, \bm{\rho}_n, \ldots}_j]/\langle \C[\underbrace{\bm{x}_n, \bm{y}_n, \bm{z}_n, \ldots}_k, \underbrace{\bm{\theta}_n, \bm{\xi}_n, \bm{\rho}_n, \ldots}_j]_+^{\mathfrak{S}_n} \rangle.\end{equation}
Commuting variables commute with all variables. Anticommuting variables anticommute with all anticommuting variables, that is, $\theta_i \theta_j = - \theta_j \theta_i$, which also implies that $\theta_i^2 = 0$. 

We briefly overview some recent results on the special cases of $R_n^{(k,j)}$ for which there has been significant progress. Kim and Rhoades \cite{KimRhoades2022} showed that the fermionic diagonal coinvariant ring $R_n^{(0,2)} = \C[\bm{\theta}_n,\bm{\xi}_n]/\allowbreak \langle \C[\bm{\theta}_n,\bm{\xi}_n]_+^{\mathfrak{S}_n} \rangle$ has dimension $\binom{2n-1}{n}$, confirming a conjecture of Zabrocki \cite{Zabrocki2020}. They gave a combinatorial formula for its Hilbert series:
\begin{equation}\Hilb(R_n^{(0,2)}; u, v) = \sum_{\pi \in \Pi(n)_{> 0}} u^{\deg_\theta(\pi)} v^{\deg_\xi(\pi)},\end{equation}
where $\Pi(n)_{> 0}$ denotes the set of modified Motzkin paths\footnote{This and other undefined terms which are needed will be defined in the subsequent sections.} of length $n$. They also gave a monomial basis for $R_n^{(0,2)}$, and found its Frobenius series \cite[Theorem 6.1]{KimRhoades2022}:
\begin{equation}\label{eq:02frob} \Frob(R_n^{(0,2)};u,v) = \sum_{i+j < n}u^iv^j( s_{(n-i,1^i)}*s_{(n-j,1^j)}-s_{(n-i+1,1^{i-1})}*s_{(n-j+1,1^{j-1})}),\end{equation}
where $*$ denotes the Kronecker product and $s_{(n+1,1^{-1})}$ is interpreted as $0$.

The superspace coinvariant ring is $R_n^{(1,1)} = \C[\bm{x}_n,\bm{\theta}_n]/\allowbreak \langle \C[\bm{x}_n,\bm{\theta}_n]_+^{\mathfrak{S}_n} \rangle$. Sagan and Swanson \cite{SaganSwanson2024} conjectured, and Rhoades and Wilson \cite{RhoadesWilson2023} proved, that its Hilbert series is
\begin{equation} \Hilb(R_n^{(1,1)}; q; u) = \sum_{k=1}^n u^{n-k} [k]_q! \Stir_q(n,k),\end{equation}
where the $q$-Stirling number $\Stir_q(n,k)$ is defined by the recurrence relation $\Stir_q(n,k) = [k]_q \Stir_q(n-1,k) + \Stir_q(n-1,k-1)$ with initial conditions $\Stir_q(0,k) = \delta_{k,0}$. 
This implies that the dimension of $R_n^{(1,1)}$ is the ordered Bell number, which counts the number of ordered set partitions of $\{1,\ldots,n\}$.
Sagan and Swanson \cite{SaganSwanson2024} conjectured and Angarone, Commins, Karn, Murai, and Rhoades \cite{Angarone2025} proved that a certain super-Artin set is a basis for $R_n^{(1,1)}$.
A Frobenius series has been conjectured for $R_n^{(1,1)}$ \cite[Equation 5.2]{BHIR}:
\begin{equation}\label{eq:11conjFrob} \Frob(R_n^{(1,1)};q;u) = \sum_{k=0}^{n-1} \sum_{\lambda \vdash n} \sum_{T \in \SYT(\lambda)} u^k q^{\maj(T) - k\des(T) + \binom{k}{2}}  \qbinom{\des(T)}{k}_q s_\lambda.\end{equation}

In general, much less is known about $R_n^{(k,j)}$ when $k+j \geq 3$. 
In 1994, Haiman \cite{Haiman1994} conjectured that the dimension of $R_n^{(3,0)}$ is given by $2^n(n+1)^{n-2}$, which is still open. A partially-graded Frobenius series has been conjectured by Bergeron and Pr\'eville-Ratelle (originally \cite{BergeronPrevilleRatelle}; phrased as in \cite{BHIR}):
\begin{equation} \Frob(R_n^{(3,0)};q,t,1) = \sum_{\alpha \preceq \beta} q^
{\dist(\alpha,\beta)}\mathbb{L}_{\beta}(t),\end{equation}
where $\alpha \preceq \beta$ are pairs of elements (Dyck paths) of the Tamari lattice, $\dist(\alpha,\beta)$ is the length of the longest chain from $\alpha$ to $\beta$, and $\mathbb{L}_\beta(t)$ denotes the LLT polynomial associated to the Dyck path $\beta$.

There is no conjecture for the dimension of $R_n^{(0,3)}$.

Zabrocki \cite{Zabrocki2019} conjectured the following Frobenius series for $R_n^{(2,1)}$: \begin{equation}\label{eq:zabrocki}
\Frob(R_n^{(2,1)};q,t;u) = \Delta'_{e_{n-1} + u e_{n-2} + \cdots + u^{n-1}}(e_n),
\end{equation} where $\Delta_f'$ is a Macdonald eigenoperator defined by $\Delta_f'\tilde{H}_\mu[X;q,t] = f[B_\mu(q,t) - 1]\tilde{H}_\mu[X;q,t]$.\footnote{For definitions of the terms in this expression, see \cite{HaglundRemmelWilson2018}.} Finally, its dimension is conjectured by Zabrocki \cite{Zabrocki2020} to be given by $\frac{1}{2}\sum_{k=0}^{n+1} \binom{n+1}{k}\frac{k^n}{n+1}$ (see OEIS sequence \href{https://oeis.org/A201595}{A201595} \cite{OEIS}). Haglund, Rhoades, and Shimozono showed that this implies the conjecture in equation~\eqref{eq:11conjFrob} \cite{HaglundRhoadesShimozono}.

D'Adderio, Iraci, and Vanden Wyngaerd \cite{DadderioIraciVandenWyngaerd2021} introduced certain symmetric function operators $\Theta_f$, called Theta operators, where $f$ is any symmetric function, in order to extend Zabrocki's conjecture on the Frobenius series of $R_n^{(2,1)}$ to a conjectural Frobenius series for $R_n^{(2,2)}$. They conjectured that
\begin{equation}\label{eq:thetaconj} \Frob(R_n^{(2,2)}; q,t;u,v) = \sum_{\substack{k,\ell \geq 0,\\ k + \ell < n}} u^k v^\ell\Theta_{e_k}\Theta_{e_\ell}\nabla e_{n-k-\ell},\end{equation}
which is known as the Theta conjecture.
Since they showed that $\Delta'_{n-k-1} e_n = \Theta_{e_k} \nabla e_{n-k}$, Zabrocki's conjecture is recovered by setting $v=0$ (equivalently, $\ell=0$).

The ring $R_n^{(1,2)} = \C[\bm{x}_n,\bm{\theta}_n,\bm{\xi}_n]/\allowbreak \langle \C[\bm{x}_n,\bm{\theta}_n,\bm{\xi}_n]_+^{\mathfrak{S}_n} \rangle$ is the main object of study in this paper. 
Zabrocki \cite{Zabrocki2020} conjectured that the dimension of $R_n^{(1,2)}$ is $2^{n-1}n!$, and an ungraded Frobenius series was conjectured by Bergeron \cite{Bergeron2020}:
\begin{equation}\label{eq:Bergeron_conj1} \Frob(R_n^{(1,2)};q;u,v)|_{q=u=v=1} = \frac{1}{2} \sum_{\mu \vdash n} 2^{\ell(\mu)}(-1)^{n-\ell(\mu)} \binom{\ell(\mu)}{d_1(\mu),\ldots,d_n(\mu)} p_\mu,\end{equation}
where $\ell(\mu)$ denotes the length of the partition $\mu$ and $d_i(\mu)$ denotes the number of parts of size $i$ in $\mu$. Another formulation of the conjecture was given in \cite{BHIR}:
\begin{equation}\label{eq:Bergeron_conj2} \Frob(R_n^{(1,2)};q;u,v)|_{q=u=v=1} = \sum_{\mu \vdash n} e_\mu \sum_{\alpha \sim \mu} \alpha_1(2\alpha_2 -1) \cdots (2\alpha_{\ell(\alpha)}-1),\end{equation}
where the second sum is over all compositions $\alpha$ whose parts rearrange to the partition $\mu$.

The Theta conjecture specialized at $t=0$ immediately gives a conjectural Frobenius series for $R_n^{(1,2)}$, however, the specialization can only be done after applying both of the Theta operators, so the resulting formula does not easily simplify.
Recently, Iraci, Nadeau, and Vanden Wyngaerd \cite{IraciNadeauVandenWyngaerd2024} have given a conjectural Hilbert series (equation~\eqref{eq:INVW_conj_Hilb}) and conjectural Frobenius series (equation~\eqref{eq:INVW_conj_Frob} and Proposition~\ref{prop:INVW-quasisymmetric_expansion}) using the combinatorics of segmented Smirnov words. However, they did not propose a monomial basis. 
Their results which will be used in this paper are surveyed in Section~\ref{sec:background-smirnov}.

Our first main contribution is a combinatorial construction of a set of monomials, denoted by $B_n^{(1,2)}$ (Definition~\ref{def:B_n^{(1,2)}}), which we conjecture to give a basis of $R_n^{(1,2)}$ (Conjecture~\ref{conj:R_n^{(1,2)}}).
If the conjecture holds, it implies the following combinatorial Hilbert series (Proposition~\ref{prop:Hilb(R_n^{(1,2)})}):
\begin{equation}\label{eq:comb_Hilbert} \Hilb(R_n^{(1,2)};q;u,v) = \sum_{\pi \in \Pi(n)_{>0}} u^{\deg_\theta(\pi)} v^{\deg_\xi(\pi)} \stair_q(\pi).\end{equation}
In support of this conjecture, we show that the cardinality of $B_n^{(1,2)}$ is $2^{n-1}n!$ (Theorem~\ref{thm:cardinality}). 
By establishing a weight-preserving bijection between $B_n^{(1,2)}$ and the set of segmented permutations $\SW(1^n)$ (Theorem~\ref{thm:bijection}), we are able to show that our conjectural Hilbert series is equivalent to the conjectural Hilbert series of Iraci, Nadeau, and Vanden Wyngaerd \cite{IraciNadeauVandenWyngaerd2024} (Corollary~\ref{cor:equivalence_Hilbert}).

Our second main contribution is a simple combinatorial formula for the conjectural Frobenius series of $R_n^{(1,2)}$ (Conjecture~\ref{conj:frob}):
\begin{equation}\Frob(R_n^{(1,2)};q;u,v) = \sum_{b \in B_n^{(1,2)}} u^{\deg_\theta(b)} v^{\deg_\xi(b)} q^{\deg_x(b)} Q_{\Asc(b),n},\end{equation}
where $Q_{S,n}$ denotes the fundamental quasisymmetric function.
We show that this is equivalent to the conjectural Frobenius series of Iraci, Nadeau, and Vanden Wyngaerd (Theorem~\ref{thm:equiv_Frobenius}).
A benefit of using $B_n^{(1,2)}$ instead of segmented permutations is that determining $\Asc(b)$ and $\deg_x(b)$ for $b \in B_n^{(1,2)}$ is typically more direct than determining $\Split(\sigma)$ and $\sminv(\sigma)$ for $\sigma \in \SW(1^n)$, which fulfill analogous roles.

In Section~\ref{sec:quasisymmetric}, we give a formula for the $m_\mu$ coefficient of our conjectural Frobenius series.
In Section~\ref{sec:hooks}, we extend some results of Iraci, Nadeau, and Vanden Wyngaerd on the sign character to hook shapes. Specifically, we give two formulas, one in terms of the proposed basis (Theorem~\ref{thm:schur-coeff-hook}) and one in terms of $q$-binomial coefficients (Theorem~\ref{thm:qbinomial}):
\begin{align*}&\langle \Frob(R_n^{(1,2)};q;u,v), s_{(d+1,1^{n-d-1})}\rangle\\ &\quad= \sum_{k+\ell < n} u^k v^\ell q^{\binom{n-d-k-\ell}{2}}\qbinom{n-1-d}{\ell}_q\qbinom{n-1-k}{d}_q\qbinom{n-1-\ell}{k}_q.\end{align*}

In Section~\ref{sec:typeB}, we consider some of the same questions in type $B$. Let \begin{equation}R_{\mathfrak{B}_n}^{(k,j)} = \C[\underbrace{\bm{x}_n, \bm{y}_n, \bm{z}_n, \ldots}_k, \underbrace{\bm{\theta}_n, \bm{\xi}_n, \bm{\rho}_n, \ldots}_j]/\langle \C[\underbrace{\bm{x}_n, \bm{y}_n, \bm{z}_n, \ldots}_k, \underbrace{\bm{\theta}_n, \bm{\xi}_n, \bm{\rho}_n, \ldots}_j]_+^{\mathfrak{B}_n} \rangle,\end{equation} where $\mathfrak{B}_n$ denotes the Weyl group of type $B$ and rank $n$ acting diagonally, that is, by permuting variables with sign.

The Hilbert series for $R_{\mathfrak{B}_n}^{(1,0)}$ is well-known to be $[2n]_q!!$, and unknown in general for $R_{\mathfrak{B}_n}^{(2,0)}$, although Gordon \cite{Gordon} constructed a notable singly-graded $\mathfrak{B}_n$-quotient module of $R_{\mathfrak{B}_n}^{(2,0)}$ with dimension $(2n+1)^n$, as conjectured by Haiman \cite{Haiman1994}. Gordon's construction shows that $\dim R_{\mathfrak{B}_n}^{(2,0)} \geq (2n+1)^n$. Equality holds only for $n = 2,3$; for $n \geq 4$, Haiman \cite{Haiman1994} conjectured and Ajila and Griffeth \cite{AjilaGriffeth} proved that $\dim R_{\mathfrak{B}_n}^{(2,0)} > (2n+1)^n$.

The Hilbert series for $R_{\mathfrak{B}_n}^{(0,2)}$ was found by Kim and Rhoades (in fact, their results were type-independent). The Hilbert series for $R_{\mathfrak{B}_n}^{(1,1)}$ was conjectured by Sagan and Swanson \cite{SaganSwanson2024} to be:
\begin{equation} \Hilb(R_{\mathfrak{B}_n}^{(1,1)}; q; u) = \sum_{k=1}^n u^{n-k} [2k]_q!! \Stir^B_q(n,k),\end{equation}
where the type $B$ $q$-Stirling number $\Stir^B_q(n,k)$ is defined by the recurrence relation $\Stir^B_q(n,k) = [2k+1]_q \Stir^B_q(n-1,k) + \Stir^B_q(n-1,k-1)$ with initial conditions $\Stir^B_q(0,k) = \delta_{k,0}$. 

We interpolate between the work of Kim--Rhoades and Sagan--Swanson to give a combinatorial construction of a set of monomials, denoted by $B_{\mathfrak{B}_n}^{(1,2)}$ (Definition~\ref{def:B_{B_n}^{(1,2)}}), which we conjecture to be a basis of $R_{\mathfrak{B}_n}^{(1,2)}$ (Conjecture~\ref{conj:R_{B_n}^{(1,2)}}).
If the conjecture holds, this implies the following combinatorial Hilbert series:
\begin{equation} \Hilb(R_{\mathfrak{B}_n}^{(1,2)}; q; u) = \sum_{\pi \in \Pi_{\geq 0}(n)} u^{\deg_\theta(\pi)} v^{\deg_\xi(\pi)} \stair_q^B(\pi).\end{equation}

We also conjecture that $\dim R_{\mathfrak{B}_n}^{(1,2)} = 4^n n!$ (Conjecture~\ref{conj:dimR_{B_n}^{(1,2)}}), and conclude by showing that the cardinality of $B_{\mathfrak{B}_n}^{(1,2)}$ agrees with this conjecture (Theorem~\ref{thm:typeB_enumeration}).

\section{Background}\label{sec:background}
In this section, we describe the Kim--Rhoades basis for $R_n^{(0,2)}$ \cite{KimRhoades2022} and the super-Artin basis for $R_n^{(1,1)}$ conjectured by Sagan--Swanson \cite{SaganSwanson2024} and proven by Angarone, Commins, Karn, Murai, and Rhoades \cite{Angarone2025}.

Define the set of \textit{modified Motzkin paths} of length $n$, $\Pi(n)_{>0}$, to be the set of all paths $\pi = (p_1,\ldots,p_n)$ in $\Z^2$ such that each step $p_i$ is one of: 
\begin{enumerate}[(a)]
    \item an up-step $(1,1)$,
    \item a horizontal step $(1,0)$ with decoration $\theta_i$,
    \item a horizontal step $(1,0)$ with decoration $\xi_i$,
    \item or a down-step $(1,-1)$ with decoration $\theta_i\xi_i$,
\end{enumerate}   
where the first step must be an up-step, and subsequently, the path never goes below the horizontal line $y=1$.\footnote{Kim and Rhoades defined the modified Motzkin paths in a slightly different manner, where there are no decorations on the down-steps, but they still contribute $\theta_i\xi_i$ to the weight. Furthermore, they defined the decorations to be just $\theta$ or $\xi$ instead of $\theta_i$ and $\xi_i$. Converting between these conventions is straightforward.}

Define the \textit{weight} $\wt(p_i)$ of a step $p_i$ of a modified Motzkin path to be its decoration, or $1$ if it does not have a decoration. Then define the \textit{weight} $\wt(\pi)$ of a modified Motzkin path $\pi \in \Pi(n)_{>0}$ to be the product of the weights of each step $p_i$, that is,
\[ \wt(\pi) := \prod_{p_i \in (p_1,\ldots,p_n) \, =\,  \pi} \wt(p_i).\]

\begin{definition}[\!\!{\cite{KimRhoades2022}}]
    The \textit{Kim--Rhoades basis} $B_n^{(0,2)}$ is the set of all weights of the modified Motzkin paths $\pi \in \Pi(n)_{>0}$, that is,
    \[ B_n^{(0,2)} := \{ \wt(\pi)\, | \, \pi \in \Pi(n)_{>0}\}.\]
\end{definition} 

For an example, see Figure~\ref{fig:Motzkin_3_example}. We are justified in calling it a basis because of the following result of Kim and Rhoades.

\begin{theorem}[\!\!\cite{KimRhoades2022}]
    The Kim--Rhoades basis $B_n^{(0,2)}$ is a basis for $R_n^{(0,2)}$.
\end{theorem}

\begin{figure}[ht]
    \centering
    \begin{tikzpicture}
        \def\sep{2.7}
        \node[anchor=south west] at (0, 0) {
            \motzkin{up,up,up}{1}
        };
        \node[anchor=south west] at (4*\sep, {0*\sep}) {
            \motzkin{up,flat-theta,up}{$\theta_2$}
        };
        \node[anchor=south west] at ({0*\sep}, {-1*\sep}) {
            \motzkin{up,flat-xi,up}{$\xi_2$}
        };
        \node[anchor=south west] at (\sep, {0*\sep}) {
            \motzkin{up,up,flat-theta}{$\theta_3$}
        };
        \node[anchor=south west] at ({2*\sep}, {0*\sep}) {
            \motzkin{up,up,flat-xi}{$\xi_3$}
        };
        \node[anchor=south west] at (1*\sep, {-1*\sep}) {
            \motzkin{up,flat-theta,flat-theta}{$\theta_2\theta_3$}
        };
        \node[anchor=south west] at ({3*\sep}, {0*\sep}) {
            \motzkin{up,up,down}{$\theta_3\xi_3$}
        };
        \node[anchor=south west] at ({3*\sep}, {-1*\sep}) {
            \motzkin{up,flat-xi,flat-theta}{$\xi_2\theta_3$}
        };
        \node[anchor=south west] at ({2*\sep}, {-1*\sep}) {
            \motzkin{up,flat-theta,flat-xi}{$\theta_2\xi_3$}
        };
        \node[anchor=south west] at ({4*\sep}, {-1*\sep}) {
            \motzkin{up,flat-xi,flat-xi}{$\xi_2\xi_3$}
        };
    \end{tikzpicture}
    \caption{The basis $B_3^{(0,2)}$. Each modified Motzkin path is labeled with its corresponding monomial.}
    \label{fig:Motzkin_3_example}
\end{figure}
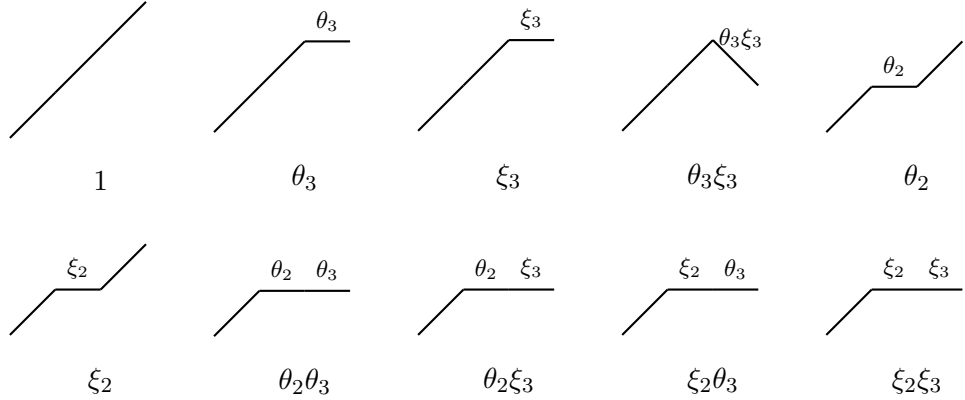

Next, we recall the super-Artin basis, defined by Sagan and Swanson. Let $\chi(P)$ be $1$ if the proposition $P$ is true, and $0$ if the proposition $P$ is false. Let $\theta_T$ denote the ordered product $\theta_{t_1} \cdots \theta_{t_k}$ for any subset $T = \{ t_1 < \cdots < t_k\} \subseteq \{1,\ldots,n\}$. For any $T \subseteq \{2,\ldots, n\}$, define the \textit{$\alpha$-sequence} $\alpha(T) = (\alpha_1(T), \ldots, \alpha_n(T))$ recursively by the initial condition $\alpha_1(T) = 0$ and for $2 \leq  i \leq n$, 
\[
    \alpha_i(T) =\alpha_{i-1}(T) + \chi(i \not \in T).
\]
Let $x^\alpha$ denote $x_1^{\alpha_1}\cdots x_n^{\alpha_n}$ for a sequence $\alpha = (\alpha_1,\ldots,\alpha_n)$.
\begin{definition}[\!\!\cite{SaganSwanson2024}]
    The super-Artin set is \[B_n^{(1,1)} := \{x^\alpha \theta_T \, | \, T \subseteq \{2,\ldots,n\} \text{ and } \alpha \leq \alpha(T) \text{ componentwise}\}.\]
\end{definition}

See Figure~\ref{fig:super_Artin_3_example} for an example. The following result was conjectured by Sagan and Swanson, and was proven by Angarone, Commins, Karn, Murai, and Rhoades.

\begin{theorem}[\!\!\cite{Angarone2025}]
    The super-Artin set $B_n^{(1,1)}$ is a basis for $R_n^{(1,1)}$.
\end{theorem}

\begin{figure}[ht]
    \centering
    \begin{tikzpicture}
        \def\sep{2.7}
        \node[anchor=south west] at (0, 0) {
            \motzkinpath{skip,skip,skip}{0,1,2}{0,0,0}{1}
        };
        \node[anchor=south west] at ({\sep}, 0) {
            \motzkinpath{skip,skip,skip}{0,1,2}{0,0,1}{$x_3$}
        };
        \node[anchor=south west] at ({2*\sep}, 0) {
            \motzkinpath{skip,skip,skip}{0,1,2}{0,1,0}{$x_2$}
        };
        \node[anchor=south west] at ({3*\sep}, 0) {
            \motzkinpath{skip,skip,skip}{0,1,2}{0,1,1}{$x_2x_3$}
        };
        \node[anchor=south west] at ({4*\sep}, 0) {
            \motzkinpath{skip,skip,skip}{0,1,2}{0,0,2}{$x_3^2$}
        };
        \node[anchor=south west] at ({0*\sep}, -1*\sep) {
            \motzkinpath{skip,skip,skip}{0,1,2}{0,1,2}{$x_2x_3^2$}
        };
        \node[anchor=south west] at (1*\sep, {-2*\sep}) {
            \motzkinpath{skip,skip,skip}{0,0,1}{0,0,0}{$\theta_2$}
        };
        \node[anchor=south west] at ({2*\sep}, {-1*\sep}) {
            \motzkinpath{skip,skip,skip}{0,1,1}{0,0,1}{$x_3\theta_3$}
        };
        \node[anchor=south west] at ({4*\sep}, {-1*\sep}) {
            \motzkinpath{skip,skip,skip}{0,1,1}{0,1,1}{$x_2x_3\theta_3$}
        };
        \node[anchor=south west] at (\sep, {-1*\sep}) {
            \motzkinpath{skip,skip,skip}{0,1,1}{0,0,0}{$\theta_3$}
        };
        \node[anchor=south west] at ({3*\sep}, {-1*\sep}) {
            \motzkinpath{skip,skip,skip}{0,1,1}{0,1,0}{$x_2\theta_3$}
        };
        \node[anchor=south west] at ({2*\sep}, {-2*\sep}) {
            \motzkinpath{skip,skip,skip}{0,0,1}{0,0,1}{$x_3\theta_2$}
        };
        \node[anchor=south west] at (3*\sep, {-2*\sep}) {
            \motzkinpath{skip,skip,skip}{0,0,0}{0,0,0}{$\theta_2\theta_3$}
        };
    \end{tikzpicture}
    \caption{The basis $B_3^{(1,1)}$. The $\alpha$-sequence is shown as the outline of all boxes, and those $x_i$ used for a particular basis element are shaded in gray.}
    \label{fig:super_Artin_3_example}
\end{figure}
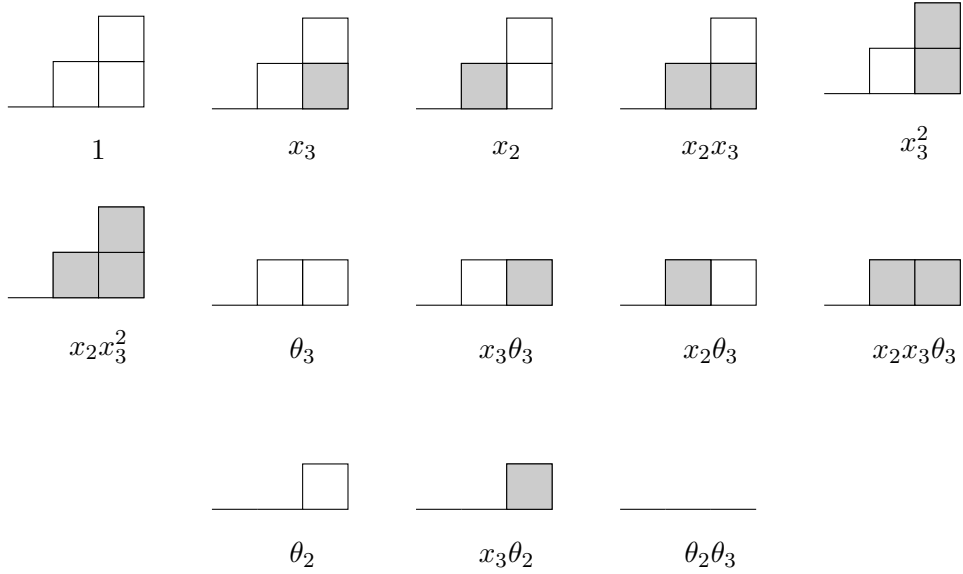

In this paper, we assume familiarity with the basics of symmetric function theory (see for example \cite[Chapter I]{Macdonald} or \cite[Chapter 7]{StanleyEC2}). Let $m_\lambda$, $e_\lambda$, $h_\lambda$, $p_\lambda$, and $s_\lambda$ denote respectively the monomial, elementary, complete homogeneous, power-sum, and Schur symmetric functions in infinitely many variables. 

Finally, we record the definitions of the multigraded Hilbert and Frobenius series of $R_n^{(k,j)}$ (see for example \cite{Bergeron2020}). Fix integers $k,j \geq 0$.
$R_n^{(k,j)}$ decomposes as a direct sum of multihomogenous components, which are $\mathfrak{S}_n$-modules:
\[R_n^{(k,j)} = \bigoplus_{r_1,\ldots,r_k, s_1, \ldots, s_j \geq 0} (R_n^{(k,j)})_{r_1,\ldots,r_k, s_1, \ldots, s_j}.\]
We denote the multigraded Hilbert series by 
\begin{align*} &\Hilb(R_n^{(k,j)}; q_1, \ldots, q_k; u_1,\ldots, u_j)\\&\quad:= \sum_{r_1,\ldots,r_k, s_1, \ldots, s_j \geq 0}\dim \left((R_n^{(k,j)})_{r_1,\ldots,r_k, s_1, \ldots, s_j} \right)q_1^{r_1} \cdots q_k^{r_k}u_1^{s_1} \cdots u_j^{s_j},\end{align*}
and the multigraded Frobenius series by 
\begin{align*} &\Frob(R_n^{(k,j)}; q_1, \ldots, q_k; u_1,\ldots, u_j)\\&\quad:= \sum_{r_1,\ldots,r_k, s_1, \ldots, s_j \geq 0} F\Char\left((R_n^{(k,j)})_{r_1,\ldots,r_k, s_1, \ldots, s_j} \right)q_1^{r_1} \cdots q_k^{r_k}u_1^{s_1} \cdots u_j^{s_j},\end{align*}
where $F$ denotes the Frobenius characteristic map and $\Char$ denotes the character.
For simplicity, if $k \leq 2$, we will use $q,t$ for $q_1,q_2$ and if $j \leq 2$, we will use $u,v$ for $u_1,u_2$. Recall that $\langle \Frob(R_n^{(k,j)}; q_1, \ldots, q_k; u_1,\ldots, u_j), h_1^n \rangle = \Hilb(R_n^{(k,j)}; q_1, \ldots, q_k; u_1,\ldots, u_j)$.

\section{The conjectural monomial basis}\label{sec:thebasis}

The goal of this section is to interpolate between the Kim--Rhoades basis and the super-Artin basis to construct a new set $B_n^{(1,2)}$, which we conjecture to be a basis for $R_n^{(1,2)}$.

We generalize the $\alpha$-sequence as follows. Let $\xi_S$ denote the ordered product $\xi_{s_1} \cdots \xi_{s_k}$ for any subset $S = \{ s_1 < \cdots < s_k\} \subseteq \{1,\ldots,n\}$. For any $T,S \subseteq \{2,\ldots,n\}$, define the \textit{generalized $\alpha$-sequence} $\alpha(T,S) = (\alpha_1(T,S), \ldots, \alpha_n(T,S))$ recursively by the initial condition $\alpha_1(T,S) = 0$ and for $2 \leq  i \leq n$, 
\[
    \alpha_i(T,S) =\alpha_{i-1}(T,S) -1 + \chi(i \not \in T) + \chi(i \not \in S).
\]

\begin{definition}\label{def:B_n^{(1,2)}}
     We let \[B_n^{(1,2)} := \{x^\alpha \theta_T \xi_S \, | \, \theta_T \xi_S \in B_n^{(0,2)} \text{ and } 0 \leq \alpha_i \leq \alpha_i(T,S) \text{ for all } i \in \{1,\ldots,n\}\}.\]
\end{definition}

See Figure~\ref{fig:Basis_3_example} for an example of $B_n^{(1,2)}$ at $n=3$.

\begin{figure}[ht]
    \centering
    \begin{tikzpicture}
        \def\sep{2.7}
        \node[anchor=south west] at (0, 0) {
            \motzkinpath{up,up,up}{0,1,2}{0,0,0}{1}
        };
        \node[anchor=south west] at ({\sep}, 0) {
            \motzkinpath{up,up,up}{0,1,2}{0,0,1}{$x_3$}
        };
        \node[anchor=south west] at ({2*\sep}, 0) {
            \motzkinpath{up,up,up}{0,1,2}{0,1,0}{$x_2$}
        };
        \node[anchor=south west] at ({3*\sep}, 0) {
            \motzkinpath{up,up,up}{0,1,2}{0,1,1}{$x_2x_3$}
        };
        \node[anchor=south west] at (0, {-1*\sep}) {
            \motzkinpath{up,up,up}{0,1,2}{0,0,2}{$x_3^2$}
        };
        \node[anchor=south west] at ({\sep}, {-1*\sep}) {
            \motzkinpath{up,up,up}{0,1,2}{0,1,2}{$x_2x_3^2$}
        };
        \node[anchor=south west] at ({2*\sep}, {-1*\sep}) {
            \motzkinpath{up,up,flat-theta}{0,1,1}{0,0,0}{$\theta_3$}
        };
        \node[anchor=south west] at ({3*\sep}, {-1*\sep}) {
            \motzkinpath{up,up,flat-theta}{0,1,1}{0,0,1}{$x_3\theta_3$}
        };
        \node[anchor=south west] at (0, {-2*\sep}) {
            \motzkinpath{up,up,flat-theta}{0,1,1}{0,1,0}{$x_2\theta_3$}
        };
        \node[anchor=south west] at ({\sep}, {-2*\sep}) {
            \motzkinpath{up,up,flat-theta}{0,1,1}{0,1,1}{$x_2x_3\theta_3$}
        };
        \node[anchor=south west] at ({2*\sep}, {-2*\sep}) {
            \motzkinpath{up,up,flat-xi}{0,1,1}{0,0,0}{$\xi_3$}
        };
        \node[anchor=south west] at ({3*\sep}, {-2*\sep}) {
            \motzkinpath{up,up,flat-xi}{0,1,1}{0,0,1}{$x_3\xi_3$}
        };  
        \node[anchor=south west] at (0, {-3*\sep}) {
            \motzkinpath{up,up,flat-xi}{0,1,1}{0,1,0}{$x_2\xi_3$}
        };
        \node[anchor=south west] at ({1*\sep}, {-3*\sep}) {
            \motzkinpath{up,up,flat-xi}{0,1,1}{0,1,1}{$x_2x_3\xi_3$}
        };
        \node[anchor=south west] at ({2*\sep}, {-3*\sep}) {
            \motzkinpath{up,up,down}{0,1,0}{0,0,0}{$\theta_3\xi_3$}
        };
        \node[anchor=south west] at ({3*\sep}, {-3*\sep}) {
            \motzkinpath{up,up,down}{0,1,0}{0,1,0}{$x_2\theta_3\xi_3$}
        };
        \node[anchor=south west] at (0, {-4*\sep}) {
            \motzkinpath{up,flat-theta,up}{0,0,1}{0,0,0}{$\theta_2$}
        };
        \node[anchor=south west] at ({\sep}, {-4*\sep}) {
            \motzkinpath{up,flat-theta,up}{0,0,1}{0,0,1}{$x_3\theta_2$}
        };    
        \node[anchor=south west] at ({2*\sep}, {-4*\sep}) {
            \motzkinpath{up,flat-xi,up}{0,0,1}{0,0,0}{$\xi_2$}
        };
        \node[anchor=south west] at ({3*\sep}, {-4*\sep}) {
            \motzkinpath{up,flat-xi,up}{0,0,1}{0,0,1}{$x_3\xi_2$}
        };
        \node[anchor=south west] at (0, {-5*\sep}) {
            \motzkinpath{up,flat-theta,flat-theta}{0,0,0}{0,0,0}{$\theta_2\theta_3$}
        };
        \node[anchor=south west] at ({\sep}, {-5*\sep}) {
            \motzkinpath{up,flat-theta,flat-xi}{0,0,0}{0,0,0}{$\theta_2\xi_3$}
        };
        \node[anchor=south west] at ({2*\sep}, {-5*\sep}) {
            \motzkinpath{up,flat-xi,flat-theta}{0,0,0}{0,0,0}{$\xi_2\theta_3$}
        };
        \node[anchor=south west] at ({3*\sep}, {-5*\sep}) {
            \motzkinpath{up,flat-xi,flat-xi}{0,0,0}{0,0,0}{$\xi_2\xi_3$}
        };
    \end{tikzpicture}
    \caption{The basis $B_3^{(1,2)}$. Below each modified Motzkin path is the outline of the generalized $\alpha$-sequence, and those $x_i$ used for a particular basis element are shaded in gray. Note that the outline of the generalized $\alpha$-sequence can be determined by the following rule: in each position $i$, the column of boxes extends up until its right side is one unit below the path above it.}
    \label{fig:Basis_3_example}
\end{figure}
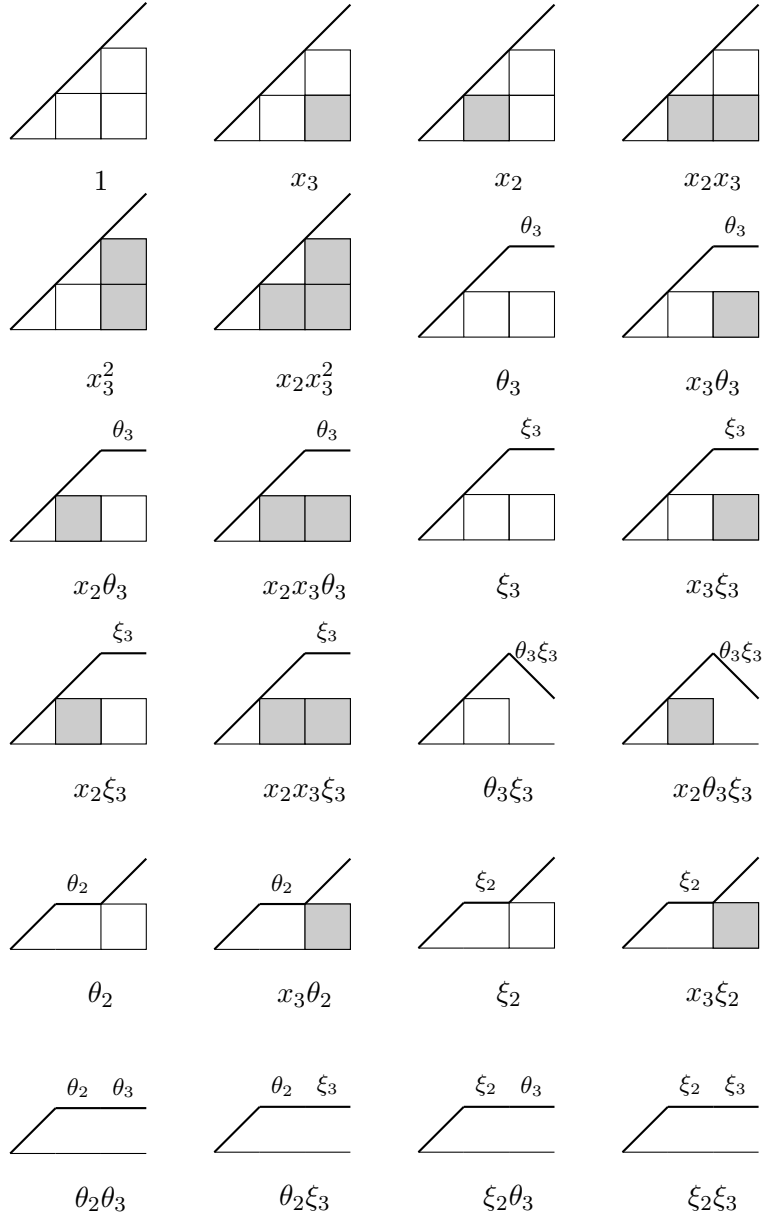

\begin{conjecture}\label{conj:R_n^{(1,2)}}
    The set $B_n^{(1,2)}$ is a basis for $R_n^{(1,2)}$.
\end{conjecture}

\begin{remark}
The conjecture has been verified for $n \leq 4$.
\end{remark}

\begin{remark}
    $B_n^{(1,2)}$ specializes to the super-Artin basis $B_n^{(1,1)}$ by setting all $\xi_i = 0$. Then the Kim--Rhoades basis elements permitted in the definition will always have $S = \varnothing$, so that condition becomes just $T \subseteq \{2,\ldots, n\}$. Thus $\alpha(T,S) = \alpha(T,\varnothing) = \alpha(T)$.
\end{remark}

\begin{remark}
    $B_n^{(1,2)}$ specializes to the Kim--Rhoades basis $B_n^{(0,2)}$ by setting all $x_i = 0$.
\end{remark}

Next, we consider the implications of Conjecture~\ref{conj:R_n^{(1,2)}} on the Hilbert series of $R_n^{(1,2)}$. To do so, define 
\[\stair_q(\pi) := \prod_{k \in \alpha(T(\pi),S(\pi))} [k+1]_q,\]
where $T(\pi)$ and $S(\pi)$ are determined by which elements in $\{2,\ldots,n\}$ appear as indices for $\theta_i$ and $\xi_i$ respectively in the weight of the modified Motzkin path $\pi$.

For a modified Motzkin path $\pi$, let $\deg_\theta(\pi) = |T(\pi)|$ and let $\deg_\xi(\pi) = |S(\pi)|$. 
\begin{proposition}\label{prop:Hilb(R_n^{(1,2)})}
    Assuming Conjecture~\ref{conj:R_n^{(1,2)}}, it follows that the Hilbert series of $R_n^{(1,2)}$ is 
    \[ \Hilb(R_n^{(1,2)};q;u,v) = \sum_{\pi \in \Pi(n)_{>0}} u^{\deg_\theta(\pi)} v^{\deg_\xi(\pi)} \stair_q(\pi).\]
\end{proposition}

\begin{proof}
The number of $\theta_i$ weights contributes to the $u$-degree and the number of $\xi_i$ weights contributes to the $v$-degree. 
Then regarding the $q$-degree, for each index $i \in \{1,\ldots, n\}$, let $k = \alpha_i(T,S)$. Then there will be $k+1$ choices for the exponent of $x_i$ in the monomial, from $0$ to $k$.
So the contribution of $x_i$ to the Hilbert series is $[k+1]_q$, and across all indices $i$, it contributes $\stair_q(\pi)$. 
Summing over the contributions of all elements $\pi \in \Pi(n)_{>0}$ gives the result.
\end{proof}

Finally, we show that the proposed basis $B_n^{(1,2)}$ has Zabrocki's conjectured dimension.

\begin{theorem}\label{thm:cardinality}
    The cardinality of $B_n^{(1,2)}$ is $2^{n-1}n!$.
\end{theorem}

\begin{proof}
    Let $P(n,r)$ denote the subset of elements of $B_n^{(1,2)}$ corresponding to modified Motzkin paths $\pi \in \Pi(n)_{>0}$ which end at height $r$. We claim that $p(n,r) := |P(n,r)| = n!\binom{n-1}{r}$, which we show by induction on $n$. The base cases consist of 
    \[ p(0,r) = \begin{cases}
        1& \text{ if } r=n,\\
        0& \text{ otherwise.}
    \end{cases}\]

    Assume the claim holds for arbitrary, fixed $n-1$ (for all $r$), and we wish to show it holds for $n$ (for all $r$). Note that $p(n,r)=0$ whenever $r < 0$ or $r > n$. 
    
    If the final step of $\pi$ is an up-step, then for each element in $P(n-1,r-1)$, we can multiply it by one of the $r+1$ choices $1, x_n, x_n^2, \ldots, x_n^r$ to get such an element in $P(n,r)$.
    
    If the final step of $\pi$ is a horizontal step, then for each element in $P(n-1,r)$, we can multiply it by one of the $r+1$ choices $1, x_n, x_n^2, \ldots, x_n^r$, and by one of the two weights $\theta_n, \xi_n$, depending on the decoration, to get such an element in $P(n,r)$.

    If the final step of $\pi$ is a down-step, then for each element in $P(n-1,r+1)$, we can multiply it by one of the $r+1$ choices $1, x_n, x_n^2, \ldots, x_n^r$, along with $\theta_n\xi_n$, to get such an element in $P(n,r)$.

    Thus
    \begin{align*}
        p(n,r) &= (r+1)p(n-1,r-1) + 2(r+1)p(n-1,r) + (r+1)p(n-1,r+1)\\
        &= (r+1)\left( (n-1)!\binom{n-2}{r-1} + 2(n-1)!\binom{n-2}{r} + (n-1)!\binom{n-2}{r+1}\right)\\
        &= (r+1)(n-1)!\binom{n}{r+1}\\
        &= n!\binom{n-1}{r},
    \end{align*}
    where we used the inductive hypothesis and some binomial coefficient identities.
    Finally, sum over all possible heights $r$ to get that $|B_n^{(1,2)}| = \sum_r n!\binom{n-1}{r} = 2^{n-1}n!$, as desired.
\end{proof}

\begin{remark}
    The argument in this enumeration is adapted from a similar argument of Corteel--Nunge~\cite[Lemma 17]{CorteelNunge2020} on marked Laguerre histories, which also are enumerated by $2^{n-1}n!$.
\end{remark}

\section{A conjectural Frobenius series}\label{sec:conj_frob}

The goal of this section is to demonstrate how the conjectural basis can be used to propose a Frobenius series for $R_n^{(1,2)}$.

Recall that the \textit{fundamental quasisymmetric function} $Q_{S,n}$ is defined by 
\[ Q_{S,n} = \sum_{\substack{a_1 \leq a_2 \leq \cdots \leq a_n,\\ a_i < a_{i+1} \text{ if } i \in S}} z_{a_1}z_{a_2}\cdots z_{a_n} ,\]
for any subset $S \subseteq \{1, \ldots, n-1\}$. 
Note that the fundamental quasisymmetric function can also be indexed by a composition $\alpha \vDash n$, which we denote by $F_\alpha$. 
To convert between the two, if $S = \{s_1 < s_2 < \cdots < s_k\}$, then use $\Comp(S) := (s_1, s_2-s_1, s_3-s_2, \ldots, s_k - s_{k-1}, n - s_k)$ for $\alpha$. 
For the other direction, if $\alpha = (\alpha_1, \alpha_2, \ldots, \alpha_\ell)$, then use $ \Set({\alpha}) := (\alpha_1, \alpha_1 + \alpha_2, \ldots, \alpha_1 + \cdots + \alpha_{\ell-1})$ for $S$ (see for example \cite[Chapter 7.19]{StanleyEC2}).

We will also need the following definitions, which are motivated by related definitions of Iraci, Nadeau, and Vanden Wyngaerd on segmented permutations. 
We will recall their definitions in Section~\ref{sec:background-smirnov} and see the relationship between their definitions and the present ones in Proposition~\ref{prop:Split_on_basis}. 
For any $b \in B_n^{(1,2)}$, we may write
\[ b = \pm \prod_{i=1}^n x_i^{\alpha_i} \theta_i^{\beta_i} \xi_i^{\gamma_i},\]
for some exponents $\alpha_i \in \Z_{\geq 0}$ and $\beta_i,\gamma_i \in \{0,1\}$.
Since this is an ordered product, reordering the factors may change the sign, however, for our purposes the sign does not matter.
Then define $i$ to be an \textit{ascent of $b$} if and only if one of the following occurs:
    \begin{itemize}
        \item $\beta_i < \beta_{i+1}$;
        \item $\beta_i = \beta_{i+1} = 1$ and $\alpha_i \geq \alpha_{i+1} + \gamma_{i+1}$; or
        \item $\beta_i = \beta_{i+1} = 0$ and $\alpha_i < \alpha_{i+1} + \gamma_{i+1}$.
    \end{itemize}
    For $b \in B_{n}^{(1,2)}$, we say that 
    \[\Asc(b) := \{ i \in \{1,\ldots,n-1\} \ | \ i \text{ is an ascent of } b\}.\]
Now we can state a conjectural Frobenius series for $R_n^{(1,2)}$ in terms of the set $B_n^{(1,2)}$.

\begin{conjecture}\label{conj:frob}
    \[\Frob(R_n^{(1,2)};q;u,v) = \sum_{b \in B_n^{(1,2)}} u^{\deg_\theta(b)} v^{\deg_\xi(b)} q^{\deg_x(b)} Q_{\Asc(b),n}.\]
\end{conjecture}

This conjecture has been verified for $n \leq 6$. We will see further evidence for this conjecture with Theorem~\ref{thm:equiv_Frobenius}.

\section{Background on segmented Smirnov words}\label{sec:background-smirnov}

The goal of this section is to provide the necessary background on the theory of segmented Smirnov words, recently advanced by Iraci, Nadeau, and Vanden Wyngaerd \cite{IraciNadeauVandenWyngaerd2024}, so that we can discuss the relationship between their work and the present work in Section~\ref{sec:smirnov}.

We recall the following definitions from \cite[Section 1]{IraciNadeauVandenWyngaerd2024}. 
A \textit{Smirnov word} of length $n$ is a word $w \in \Z_{\geq 0}^n$ such that $w_i \neq w_{i+1}$ for all $i \in \{1,\ldots, n-1\}$. 
A \textit{segmented Smirnov word} of shape $\alpha = (\alpha_1,\ldots, \alpha_s) \vDash n$ is a word $w \in \Z_{\geq 0}^n$ such that $w = w^{(1)} \cdots w^{(s)}$ is a concatenation of $s$ Smirnov words, where each subword $w^{(i)}$, called a \textit{block}, is a Smirnov word of length $\alpha_i$.  
We write segmented Smirnov words with a vertical bar between each block. 
For example, $121|13$ is a segmented Smirnov word of shape $(3,2)$.
Let $\SW(n)$ denote the set of segmented Smirnov words of length $n$. 
Let $\mu \vDash_0 n$ denote that $\mu$ is a weak composition of $n$, i.e., $\mu = (\mu_1,\mu_2, \ldots)$, where $\mu_i \in \Z_{\geq 0}$ and $\sum_{i \in \Z_{\geq 0}} \mu_i = n$. 
Given $\mu \vDash_0 n$, let $\SW(\mu)$ denote the set of segmented Smirnov words with content $\mu$, i.e., there are $\mu_1$ many 1's, $\mu_2$ many 2's, etc.

An \textit{ascent}\footnote{It should be clear from the context whether we are referring to the ascents of a (segmented) Smirnov word, or the ascents of an element $b \in B_n^{(1,2)}$, which are different.} of a Smirnov word $w$ is an index $i$ such that $w_i < w_{i+1}$ and a \textit{descent} of a Smirnov word $w$ is an index $i$ such that $w_i > w_{i+1}$.
For segmented Smirnov words, ascents and descents are those that occur strictly within its blocks.
Let $\SW(n,k,\ell)$ denote the set of segmented Smirnov words of length $n$ with exactly $k$ ascents and $\ell$ descents. Since each index $i \in \{1,\ldots,n\}$ is exactly one of: the index of an ascent, the index of a descent, or the last index in a block, we have that the number of blocks is $n-k-\ell$.
Finally, for $\mu \vDash_0 n$, let $\SW(\mu,k,\ell) := \SW(\mu) \cap \SW(n,k,\ell)$.

\begin{definition}[\!\!{\cite[Definition 1.7]{IraciNadeauVandenWyngaerd2024}}]\label{def:sminversion}
    For a segmented Smirnov word $w$, the ordered pair $(i,j)$ with $1 \leq i < j \leq n$ is a \textit{sminversion} (short for Smirnov inversion) if $w_i > w_j$ and one of the following conditions holds:
\begin{enumerate}[(1)]
    \item $w_j$ is the first letter of its block;
    \item $w_{j-1} > w_i$;
    \item $i \neq j-1$, $w_{j-1} = w_i$, and $w_{j-1}$ is the first letter of its block; or
    \item $i \neq j-1$ and $w_{j-2} > w_{j-1} = w_i$.
\end{enumerate}
\end{definition}
Denote $\SF(n,k,\ell) := (\Theta_{e_k}\Theta_{e_\ell}\nabla e_{n-k-\ell} )|_{t=0}$.\footnote{In \cite{IraciNadeauVandenWyngaerd2024}, $\SF(n,k,\ell)$ is first defined by $(\Theta_{e_k}\Theta_{e_\ell}\tilde{H}_{n-k-\ell} )|_{t=0}$, where $\tilde{H}_{n-k-\ell}$ is the modified Macdonald polynomial. These two definitions are equivalent by \cite[Lemma 3.6]{IraciRhoadesRomero}.}
Define the statistic $\sminv(w)$ on a segmented Smirnov word to be the number of sminversions of $w$.
Denote 
\[ \SW_{z;q}(n,k,\ell) := \sum_{w \in \SW(n,k,\ell)} q^{\sminv(w)}z_w.\]
Now we can state one of the main results of \cite{IraciNadeauVandenWyngaerd2024}.

\begin{theorem}[\!\!{\cite[Theorem 3.1]{IraciNadeauVandenWyngaerd2024}}]\label{thm:INVW-3.1}
    For any $ 0 \leq k+\ell < n$, 
    \[\SF(n,k,\ell) = \SW_{z;q}(n,k,\ell).\]
\end{theorem}

By summing over all possible $k,\ell$, we get their conjectural Frobenius series for $R_n^{(1,2)}$:
\begin{equation}\label{eq:INVW_conj_Frob}
    \sum_{k + \ell <  n} u^k v^\ell (\Theta_{e_k}\Theta_{e_\ell}\nabla e_{n-k-\ell} )|_{t=0} = \sum_{k + \ell <  n} u^k v^\ell \sum_{w \in \SW(n,k,\ell)} q^{\sminv(w)}z_w.
\end{equation} 

By taking the inner product of this conjectural Frobenius series with $h_1^n$, Iraci, Nadeau, and Vanden Wyngaerd are able to derive the following recursive formula for the conjectural Hilbert series of $R_n^{(1,2)}$. 
When $\mu = (1^n)$, we call the corresponding segmented Smirnov words \textit{segmented permutations}, since they have each of $\{1,\ldots,n\}$ appearing exactly once.
They also note that when we restrict to the case of segmented permutations, only conditions $(1)$ and $(2)$ in the definition of sminversion (Definition~\ref{def:sminversion}) will apply. 
Let \[\SW_q(\mu,k,\ell) := \sum_{w \in \SW(\mu,k,\ell)} q^{\sminv(w)}.\]

\begin{proposition}[\!\!{\cite[Proposition 3.2]{IraciNadeauVandenWyngaerd2024}}]\label{prop:recurrence}
    For any $k+\ell < n$, the polynomials $\SW_q(1^n,k,\ell) \in \Z[q]$ satisfy the recursion
    \begin{align*}
        \SW_q(1^n,k,\ell) &= [n-k-\ell]_q \big( \SW_q(1^{n-1},k,\ell) + \SW_q(1^{n-1},k,\ell-1)\\
        &\quad+ \SW_q(1^{n-1},k-1,\ell) + \SW_q(1^{n-1},k-1,\ell-1)\big),
    \end{align*} 
    with initial conditions $\SW_q(\varnothing,k,\ell) = \delta_{k,0}\delta_{\ell,0}$.
    Furthermore, $\langle \SF(n,k,\ell), h_1^n \rangle = \SW_q(1^n,k,\ell)$. 
\end{proposition}

To obtain a conjectural Hilbert series of $R_n^{(1,2)}$, take the sum
\begin{equation}\label{eq:INVW_conj_Hilb}
\sum_{k+\ell <n}\langle \SF(n,k,\ell), h_1^n \rangle u^k v^\ell = \sum_{k+\ell <n}\SW_q(1^n,k,\ell) u^k v^\ell.
\end{equation}

We recall some additional definitions from \cite[Section 5]{IraciNadeauVandenWyngaerd2024}. Given a segmented Smirnov word $w$, for each index $i \in \{1,\ldots,n\}$, we say:
\begin{itemize}
    \item $i$ is \textit{thick} if $i$ is initial (in a block) or $w_{i-1} > w_i$;
    \item $i$ is \textit{thin} if $i$ is not initial and $w_{i-1} < w_i$.
\end{itemize}
Note that $i$ is thick is equivalent to $i$ is initial or $i$ is the end of a descent. 
Also $i$ is thin is equivalent to $i$ is not initial and $i$ is the end of an ascent.
Let $\sigma$ be a segmented permutation of length $n$, $\sigma_i <n$, and let $j$ be defined by $\sigma_j = \sigma_i + 1$. If any of the following hold, we say that $\sigma_i$ is \textit{splitting} for $\sigma$:
\begin{enumerate}[(a)]
    \item $i$ is thick and $j$ is thin;
    \item $i$ and $j$ are both thin and $i < j$;
    \item $i$ and $j$ are both thick and $j < i$.
\end{enumerate}
Define $\Split(\sigma) = \{m \in \{1,\ldots,n-1\} \ | \ m \text{ is splitting for } \sigma\}$. 
Then $\Split$ can be used to give the following quasisymmetric expansion.

\begin{proposition}[\!\!{\cite[Proposition 5.3]{IraciNadeauVandenWyngaerd2024}}]\label{prop:INVW-quasisymmetric_expansion}
    \[\SW_{z;q}(n,k,\ell) = \sum_{\sigma \in \SW(1^n, k,\ell)} q^{\sminv(\sigma)} Q_{\Split(\sigma),n}.\]
\end{proposition}

\section{The proposed basis and segmented permutations}\label{sec:smirnov}

In this section, we establish a bijection between our proposed basis $B_n^{(1,2)}$ and the set of segmented permutations $\SW(1^n)$ which is $q,u,v$-weight preserving.
This implies the equivalence of two conjectural Hilbert series for $R_n^{(1,2)}$ (equations~\eqref{eq:comb_Hilbert} and~\eqref{eq:INVW_conj_Hilb}).

Define a map $\psi: B_n^{(1,2)} \longrightarrow \SW(1^n)$ as follows. For any $b \in B_n^{(1,2)}$, we can write it as
\[ b = \pm \prod_{i=1}^n x_i^{\alpha_i} \theta_i^{\beta_i} \xi_i^{\gamma_i},\]
for some $\alpha_i \in \Z_{\geq 0}$ and $\beta_i,\gamma_i \in \{0,1\}$. 
Each $b \in B_n^{(1,2)}$ is associated with a modified Motzkin path $\pi$.
Each $\pi$ starts with an up-step: this implies that $\alpha_1 = \beta_1 = \gamma_1 = 0$, so we start by writing the corresponding segmented permutation as $1$. 
Then, as $i$ ranges from $2$ up through $n$, do exactly one of the following for each $i$:
\begin{enumerate}[(a)]
    \item if $\beta_i = \gamma_i = 0$ (this corresponds to an up-step): insert ``$|i$'' or ``$i|$'' in such a way as to create a new block consisting of only $i$ at position $\alpha_i + 1$ from the rightmost block in the permutation (indexing starting at 1);
    \item if $\beta_i = 1$ and $\gamma_i = 0$ (this corresponds to a horizontal step with decoration $\theta_i$): insert $i$ as the last element of an existing block, at position $\alpha_i + 1$ from the rightmost block in the permutation;
    \item if $\beta_i = 0$ and $\gamma_i = 1$ (this corresponds to a horizontal step with decoration $\xi_i$): insert $i$ as the first element of an existing block, at position $\alpha_i + 1$ from the rightmost block in the permutation;
    \item if $\beta_i = \gamma_i = 1$ (this corresponds to a down-step with decoration $\theta_i\xi_i$): insert $i$ to replace a ``$|$'' and thus merge two adjacent blocks into one block, which is now at position $\alpha_i + 1$ from the rightmost block in the permutation.
\end{enumerate}
Upon completion of this process, observe that the output is some segmented permutation $\sigma$ in $\SW(1^n)$. Also, note that the height of the path is equal to the number of blocks in the corresponding segmented permutation.

\begin{example}
    Consider $b = x_3\theta_3\theta_4\xi_4\xi_5 \in B_5^{(1,2)}$. We rewrite this as $b = \pm \prod_{i=1}^n x_i^{\alpha_i} \theta_i^{\beta_i} \xi_i^{\gamma_i}$, with $(\alpha_1,\beta_1,\gamma_1) = (0,0,0)$, $(\alpha_2,\beta_2,\gamma_2) = (0,0,0)$, $(\alpha_3,\beta_3,\gamma_3) = (1,1,0)$, $(\alpha_4,\beta_4,\gamma_4) = (0,1,1)$, $(\alpha_5,\beta_5,\gamma_5) = (0,0,1)$. We start building the segmented permutation by writing 1. Then, for $i=2$, we are in case (a) and insert $|2$ to get $1|2$ so that the new block is in position 1 from the right. Next, for $i=3$, we are in case (b) and insert $3$ as the last element in the existing block in position 2 from the right, yielding $13|2$. Next, for $i=4$, we are in case (d) and insert $4$ to replace the ``$|$'' to merge the two blocks to create just one block in position 1 from the right, yielding $1342$. Finally, for $i=5$, we are in case (c) and insert $5$ as the first element in the existing block in position 1 from the right, yielding $51342$.
\end{example}

For more examples, see the tables in Appendix~\ref{app:A}.

\begin{theorem}\label{thm:bijection}
    The map $\psi : B_n^{(1,2)} \longrightarrow \SW(1^n)$ is a $q,u,v$-weight preserving bijection.
\end{theorem}

\begin{proof}
First, we can check that $\psi$ is an injection by exhibiting a left inverse. 
In other words, we will show that $B_n^{(1,2)}$ is in bijection with a subset of $\SW(1^n)$. 
Consider any $\sigma \in \psi (B_n^{(1,2)}) \subseteq \SW(1^n)$. 
We will algorithmically construct a $b \in B_n^{(1,2)}$ such that $\psi(b) = \sigma$. 
Initialize the value of $b$ at $1$. 
Each iteration of the algorithm will update the values of $\sigma$ and $b$; the values of $b$ and $\sigma$ which satisfy $\psi(b) = \sigma$ will be the initial value of $\sigma$ and the final value of $b$.
Take $i$ to be the highest number present in the segmented word $\sigma$ ($i$ will first be $n$ and then decrease by $1$ for each iteration of the algorithm).
We determine how $i$ must have been inserted when $\sigma$ was constructed from an element of $B_n^{(1,2)}$:
\begin{enumerate}[(a)]
    \item if $i$ occurs alone in a block, then $\beta_i = \gamma_i = 0$, and $\alpha_i$ is determined by the block which consists of $i$ being $\alpha_i + 1$ from the rightmost block in the permutation (indexing starting at 1);
    \item if $i$ occurs as the last element of a block with at least two elements, then $\beta_i = 1$, $\gamma_i = 0$, and $\alpha_i$ is determined by the block which contains $i$ being $\alpha_i + 1$ from the rightmost block in the permutation;
    \item if $i$ occurs as the first element of a block with at least two elements, then $\beta_i = 0$, $\gamma_i = 1$, and $\alpha_i$ is determined by the block which contains $i$ being $\alpha_i + 1$ from the rightmost block in the permutation;
    \item if $i$ occurs within a block with at least one element to the right and to the left, then $\beta_i = \gamma_i = 1$, and $\alpha_i$ is determined by the block which contains $i$ being $\alpha_i + 1$ from the rightmost block in the permutation.
\end{enumerate}
Having found the values of $\alpha_i, \beta_i,\gamma_i$, multiply the current value of $b$ by $x_i^{\alpha_i}\theta_i^{\beta_i}\xi_i^{\gamma_i}$, updating it with the new value. Then update $\sigma$ by undoing the insertion of $i$ according to the same cases:
\begin{enumerate}[(a)]
    \item delete $i$ and if two $|$ become adjacent, delete one of them;
    \item delete $i$;
    \item delete $i$;
    \item replace $i$ with a $|$.
\end{enumerate}
Then the highest value in $\sigma$ is now $i-1$.
Repeat this process until the segmented word $\sigma$ becomes $1$, and then $b$ is fully constructed. 

By Theorem~\ref{thm:cardinality}, the cardinality of $B_n^{(1,2)}$ is $2^{n-1}n!$. 
To enumerate the segmented permutations $\sigma \in \SW(1^n)$, consider that there are $n!$ permutations and $n-1$ binary choices for whether or not there is a vertical bar between any two adjacent letters.
Thus the cardinality of $\SW(1^n)$ is also $2^{n-1}n!$, and $\psi$ is a bijection, and the left inverse just described is an inverse function.

It only remains to check that $\psi$ is weight-preserving. 
We show that the $\theta$-degree and the $\xi$-degree is preserved. Consider the four cases used in the construction of a segmented permutation under $\psi$ from $b \in B_n^{(1,2)}$:
\begin{enumerate}[(a)]
    \item no ascents or descents are created, so the $\theta$-degree and the $\xi$-degree are unchanged;
    \item one ascent is created (which ends with $i$) and no descents are created, so the $\theta$-degree increases by 1 and the $\xi$-degree is unchanged;
    \item no ascents are created and one descent is created (which begins with $i$), so the $\theta$-degree is unchanged and the $\xi$-degree increases by 1;
    \item one ascent is created and one descent is created (the ascent ends with $i$ and the descent begins with $i$, as $i$ is the highest number in the segmented permutation at this point), so the $\theta$-degree increases by 1 and the $\xi$-degree increases by 1.
\end{enumerate}
In each case, the number of ascents remains the same as the $\theta$-degree and the number of descents remains the same as the $\xi$-degree.
Regarding the $x$-degree, when $i$ is inserted in the construction, it creates a sminversion with exactly all initial elements of blocks to its right, which corresponds exactly to the increase in $x$-degree by $\alpha_i$ (cf. \cite[proof of Proposition 3.2]{DadderioIraciVandenWyngaerd2021}). 
\end{proof}

We record an important consequence of this bijection, which follows since $q$-weights are preserved.

\begin{corollary}\label{cor:sminv-x-degree}
    Let $b \in B_n^{(1,2)}$. Then $\deg_x(b) = \sminv(\psi(b))$.
\end{corollary}

As another consequence of this bijection, we are able to conclude that the conjectural Hilbert series of Iraci, Nadeau, and Vanden Wyngaerd is equal to ours.
\begin{corollary}\label{cor:equivalence_Hilbert}
    \[\sum_{k+\ell <n} u^k v^\ell \sum_{w \in \SW(1^n,k,\ell)} q^{\sminv(w)} = 
    \sum_{\pi \in \Pi(n)_{>0}} u^{\deg_\theta(\pi)} v^{\deg_\xi(\pi)} \stair_q(\pi).\]
\end{corollary}
This result can also be shown directly, by showing that \[\sum_{\pi \in \Pi(n)_{>0}} u^{\deg_\theta(\pi)} v^{\deg_\xi(\pi)} \stair_q(\pi)\]
satisfies the recurrence relation given in Proposition~\ref{prop:recurrence}. 

With the bijection $\psi : B_n^{(1,2)} \longrightarrow \SW(1^n)$ in hand, the following result relates $\Asc(b)$ with $\Split(\sigma)$.

\begin{proposition}\label{prop:Split_on_basis}
For any $b \in B_n^{(1,2)}$,
    \[\Asc(b) = \Split(\psi(b)).\]
\end{proposition}

\begin{proof} We wish to relate $\Asc(b)$ for elements $b \in B_n^{(1,2)}$ with $\Split(\sigma)$ for $\psi(b) = \sigma \in \SW(1^n)$. Both definitions reduce to determining which indices $m$ are ascents of $b$ or splitting for $\sigma$, respectively. For $\sigma$, this depends on whether indices are thick or thin, so we wish to determine the analogous conditions on $b$. Consider each of the four ways to insert $m$ into a segmented permutation in $\SW(1^{m-1})$ to yield one in $\SW(1^{m})$, and how each affects the modified Motzkin path used to construct the proposed basis elements:
\begin{enumerate}[(a)]
    \item (``$|m$'' or ``$m|$'' is inserted to create a new block $\longleftrightarrow$ append up-step): this will be thick since it is the beginning of a block;
    \item ($m$ is inserted as the last element of an existing block $\longleftrightarrow$ append horizontal step with decoration $\theta_m$): this will be thin since it is preceded by a smaller number;
    \item ($m$ is inserted as the first element of an existing block $\longleftrightarrow$ append horizontal step with decoration $\xi_m$): this will be thick since it is the beginning of a block (and note that the next index remains thick since it was previously the first element in a block and is now the end of a descent);
    \item ($m$ is inserted to replace a ``$|$'' $\longleftrightarrow$ append down-step with decoration $\theta_m\xi_m$): this will be thin, as it is the end of an ascent, as the number which precedes it must be smaller (and note that the next index remains thick since it was previously the first element in a block and is now the end of a descent).
\end{enumerate}
From this discussion, we conclude that thin indices $m$ in a segmented permutation $\sigma$ correspond exactly to the weights $\theta_{\sigma_m}$ which appear as factors in $b \in B_n^{(1,2)}$.

Now we desire a specific classification for when $m = \sigma_i$ is splitting. Let $j$ be defined by $\sigma_j = \sigma_i + 1$. We have that:
\begin{itemize}
    \item if $\theta_{m} \not\in b$ ($i$ thick) and $\theta_{m+1} \in b$ ($j$ thin), then $m$ is splitting for $\sigma$;
    \item if $\theta_{m},\theta_{m+1} \in b$ ($i,j$ thin) and $ i < j$, then $m$ is splitting for $\sigma$;
    \item if $\theta_{m},\theta_{m+1} \not\in b$ ($i,j$ thick) and $ i > j$, then $m$ is splitting for $\sigma$;
\end{itemize}
and otherwise, $m$ is not splitting for $\sigma$. 
However, we have not yet described what $i$ and $j$ are without appealing to the segmented permutation. 
But we only need the relative position of $i$ and $j$ for this purpose, which we can deduce directly from $b \in B_n^{(1,2)}$, which we write as
\[ b = \pm \prod_{m=1}^n x_m^{\alpha_m} \theta_m^{\beta_m} \xi_m^{\gamma_m},\]
for some $\alpha_m \in \Z_{\geq 0}$ and $\beta_m,\gamma_m \in \{0,1\}$.

Each $\alpha_m$ indicates that when building up the corresponding segmented permutation, just after $m$ has been inserted, it is in block $\alpha_m + 1$ from the right.
If step $m$ in the path associated to $b$ is:
\begin{enumerate}[(a)]
    \item an up-step: if $\alpha_{m} \geq \alpha_{m+1}$, then $i < j$, else $i > j$;
    \item a horizontal step with decoration $\theta_m$: if $\alpha_{m} \geq \alpha_{m+1}$, then $i < j$, else $i > j$;
    \item a horizontal step with decoration $\xi_m$: if $\alpha_{m} > \alpha_{m+1}$, then $i < j$, else $i > j$;
    \item a down-step with decoration $\theta_m\xi_m$: if $\alpha_{m} > \alpha_{m+1}$, then $i < j$, else $i > j$.
\end{enumerate}
We determined whether the condition $\alpha_{m} \geq \alpha_{m+1}$ or $\alpha_{m} > \alpha_{m+1}$ is used in each case depending on how the insertion affects the relative order of $m$ and $m+1$.

Recall that for $b \in B_{n}^{(1,2)}$, $\Asc(b) = \{ m \in \{1,\ldots,n-1\} \ | \ m \text{ is an ascent of } b\}$.
At this point, we have that $m$ is splitting for $\sigma$ if and only if one of the following occurs:
    \begin{itemize}
        \item if $\theta_{m} \nmid b$ and $\theta_{m+1} \mid b$;
        \item if $\theta_{m},\theta_{m+1} \mid b$ and 
        \begin{itemize}
            \item if $b$ has an up-step or horizontal step with decoration $\theta_m$ in position $m$, we have $\alpha_{m} \geq \alpha_{m+1}$, or
            \item if $b$ has a horizontal step with decoration $\xi_m$ or a down-step with decoration $\theta_m \xi_m$ in position $m$, we have $\alpha_{m} > \alpha_{m+1}$;
        \end{itemize}
        \item if $\theta_{m},\theta_{m+1} \nmid b$ and 
        \begin{itemize}
            \item if $b$ has an up-step or horizontal step with decoration $\theta_m$, we have $\alpha_{m} < \alpha_{m+1}$, or
            \item if $b$ has a horizontal step with decoration $\xi_m$ or a down-step with decoration $\theta_m \xi_m$, we have $\alpha_{m} \leq \alpha_{m+1}$.
        \end{itemize}
    \end{itemize}
Finally, we can convert this to a criteria on only the exponents of $b$. That is, $m$ is splitting for $\sigma$ if and only if one of the following occurs:
    \begin{itemize}
        \item $\beta_m < \beta_{m+1}$;
        \item $\beta_m = \beta_{m+1} = 1$ and $\alpha_m \geq \alpha_{m+1} + \gamma_{m+1}$; or
        \item $\beta_m = \beta_{m+1} = 0$ and $\alpha_m < \alpha_{m+1} + \gamma_{m+1}$,
    \end{itemize}
    which is exactly the definition of when $m$ is an ascent of $b$, as desired.
\end{proof}

Now, we are able to establish the equivalence of our conjectural Frobenius series (Conjecture~\ref{conj:frob}) with that of Iraci, Nadeau, and Vanden Wyngaerd (equation~\eqref{eq:INVW_conj_Frob}).

\begin{theorem}\label{thm:equiv_Frobenius}
    \[\sum_{b \in B_n^{(1,2)}} u^{\deg_\theta(b)} v^{\deg_\xi(b)} q^{\deg_x(b)} Q_{\Asc(b),n} = \sum_{ k+\ell < n} u^k v^\ell \sum_{\sigma \in \SW(1^n, k,\ell)} q^{\sminv(\sigma)} Q_{\Split(\sigma),n}.\]
\end{theorem}

\begin{proof}
This follows from Theorem~\ref{thm:bijection}, establishing that $\psi: B_n^{(1,2)} \longrightarrow \SW(1^n)$ is a bijection, along with Corollary~\ref{cor:sminv-x-degree} and Proposition~\ref{prop:Split_on_basis}, which state, respectively, that if $\psi(b) = \sigma$, then $\deg_x(b) = \sminv(\sigma)$ and $\Asc(b) = \Split(\sigma)$.
\end{proof}

We conclude this section with the discussion of some specializations.
Iraci, Rhoades, and Romero \cite[Theorem 1.3]{IraciRhoadesRomero} showed that
\[ \Frob(R_n^{(0,2)};u,v) = \sum_{k+\ell < n} u^k v^\ell \Theta_{e_k} \Theta_{e_\ell} \nabla e_{n-k-\ell}|_{q=t=0}.\]
As discussed in \cite[Section 6.3]{IraciNadeauVandenWyngaerd2024}, the conjectural Frobenius for $R_n^{(1,2)}$ in equation~\eqref{eq:INVW_conj_Frob} recovers the Frobenius series for $R_n^{(0,2)}$ (equation~\eqref{eq:02frob}) by further specialization of the Theta conjecture at $q=0$. This implies by Theorem~\ref{thm:equiv_Frobenius} that Conjecture~\ref{conj:frob} also specializes to the Frobenius series for $R_n^{(0,2)}$.

The conjecture in equation~\eqref{eq:11conjFrob} can be recovered from Conjecture~\ref{conj:frob}. The Theta conjecture (equation~\eqref{eq:thetaconj}) can be specialized at $v=0$ to Zabrocki's conjecture (equation~\eqref{eq:zabrocki}), which can be further specialized at $t=0$ to equation~\eqref{eq:11conjFrob}. That is,
\begin{align}\label{eq:theta_11}
    &\sum_{k+\ell < n} u^k v^\ell \Theta_{e_k} \Theta_{e_\ell} \nabla e_{n-k-\ell}|_{v=0, t=0}\\ &\quad= \sum_{k=0}^{n-1} \sum_{\lambda \vdash n} \sum_{T \in \SYT(\lambda)} u^k q^{\maj(T) - k\des(T) + \binom{k}{2}}  \qbinom{\des(T)}{k}_q s_\lambda.\nonumber
\end{align}
On the other hand, starting with the Theta conjecture and specializing at $t=0$ gives Conjecture~\ref{conj:frob}; upon further specialization at $v=0$, since the order of specialization commutes, implies that equation~\eqref{eq:theta_11} is also equal to 
\[\sum_{b \in B_n^{(1,2)}} u^{\deg_\theta(b)} v^{\deg_\xi(b)} q^{\deg_x(b)} Q_{\Asc(b),n}|_{v=0} = \sum_{b \in B_n^{(1,1)}} u^{\deg_\theta(b)} q^{\deg_x(b)} Q_{\Asc(b),n}|_{v=0},\]
where $\Asc(b)$ for $b \in B_n^{(1,1)}$ is determined in the same way as for $b \in B_n^{(1,2)}$.

It would be interesting to see if the conjectures in equations \eqref{eq:Bergeron_conj1} or \eqref{eq:Bergeron_conj2}, can be recovered by specializing from either form of the conjectural Frobenius series given in Theorem~\ref{thm:equiv_Frobenius}. It would also be interesting to have direct combinatorial proofs of the two specializations just discussed.

\section{Refining the fundamental quasisymmetric function}\label{sec:quasisymmetric}

In this section, we establish a formula for $\langle \SF(n,k,\ell), h_\mu \rangle = [ m_\mu ] \SF(n,k,\ell)$. 
Recall from Section~\ref{sec:conj_frob} that $\Set(\mu) = \{\mu_1,\mu_1+\mu_2, \ldots, \mu_1 + \cdots + \mu_{\ell(\mu) -1}\}$, where $\ell(\mu)$ is the length of the partition $\mu$.

\begin{theorem}\label{thm:quasisymmetric-refinement-h_mu-basis} Let $\mu \vdash n$. For any fixed $k, \ell$, we have that
    \[ \langle \SF(n,k,\ell), h_\mu \rangle = \sum_{\substack{b \in B_n^{(1,2)},\\
    \deg_\theta(b) = k,\\
    \deg_\xi(b) = \ell,\\
    \Asc(b) \subseteq \Set(\mu)}} q^{\deg_x(b)}.\]
\end{theorem}

We make use of the following lemma, which is a modification of \cite[Lemma 6.14]{Haglund2008}, a similar result on (non-segmented) permutations, and is proven in the same manner.

\begin{lemma}\label{lem:quasi-h_mu}
    Let $\mu \vdash n$.
    If $f$ is a symmetric function, then for any family of constants $c(b)$, and any set $B$ such that elements $b \in B$ have a function $\Asc: B \longrightarrow \{1,\ldots, n-1\}$,
    \begin{equation}\label{eq:quasi} 
    f = \sum_{b \in B} c(b) Q_{\Asc(b), n}\end{equation}
    if and only if
    \begin{equation}\label{eq:h_mu}
    \langle f, h_\mu \rangle = \sum_{\substack{b \in B,\\ \Asc(b) \subseteq \Set(\mu)}} c(b).\end{equation}
\end{lemma}

\begin{proof}[Proof of Theorem \ref{thm:quasisymmetric-refinement-h_mu-basis}]
    Fix $k$ and $\ell$.
    Let $c(b) = q^{\deg_x(b)}$. Then we can use \[ f = \sum_{\substack{b \in B_n^{(1,2)},\\
    \deg_\theta(b) = k,\\
    \deg_\xi(b) = \ell}} Q_{\Asc(b),n} \] in Lemma~\ref{lem:quasi-h_mu}. 
    It follows that
    \[ \langle f, h_\mu \rangle = \sum_{\substack{b \in B_n^{(1,2)},\\
    \deg_\theta(b) = k,\\
    \deg_\xi(b) = \ell,\\
    \Asc(b) \subseteq \Set(\mu)}} q^{\deg_x(b)}.\]
    Proposition~\ref{prop:INVW-quasisymmetric_expansion} and Theorem \ref{thm:INVW-3.1} imply $\SF(n,k,\ell) = \sum_{\sigma \in \SW(1^n,k,\ell)} q^{\sminv(\sigma)}Q_{\Split(\sigma),n}$.
    We can restrict the bijection $\psi$ from Theorem~\ref{thm:bijection} to one on $\{b \in B_n^{(1,2)} \ | \  \deg_\theta(b) = k \text{ and } \deg_\xi(b) = \ell \}$ and $\SW(1^n,k,\ell)$ since it is $q,u,v$-weight preserving. 
    Corollary~\ref{cor:sminv-x-degree} tells us that $\deg_x(b) = \sminv(\psi(b))$ and Proposition~\ref{prop:Split_on_basis} tells us that $\Asc(b) = \Split(\psi(b))$, so $f = \SF(n,k,\ell)$, and the claim follows.
\end{proof}

In the special case where $\mu = (d+1,1^{n-d-1})$ is a hook shape, we have the following result. Denote by $B_{n}^{(1,2)}|_{(d+1)-\up}$ the subset of $b \in B_{n}^{(1,2)}$ where the corresponding modified Motzkin path begins with $d+1$ up-steps, with no $x$-degree contribution from those steps. Equivalently, if $b = \pm \prod_{m=1}^n x_m^{\alpha_m} \theta_m^{\beta_m} \xi_m^{\gamma_m}$, then $\alpha_m=\beta_m=\gamma_m=0$ for all $m \in \{1,\ldots,d+1\}$.

\begin{corollary}\label{cor:h_hook}
    \[\langle \SF(n,k,\ell), h_{(d+1, 1^{n-d-1})}\rangle = \sum_{\substack{b \in B_{n}^{(1,2)}|_{(d+1)-\up}, \\ \deg_\theta(b) = k,\\ \deg_\xi(b) = \ell.}} q^{\deg_x(b)}. \]
\end{corollary}

\begin{proof} From Theorem~\ref{thm:quasisymmetric-refinement-h_mu-basis} specialized to hook shapes, we have that
    \[\langle \SF(n,k,\ell), h_{(d+1, 1^{n-d-1})}\rangle = \sum_{\substack{b \in B_{n}^{(1,2)}, \\ \deg_\theta(b) = k,\\ \deg_\xi(b) = \ell,\\ \Asc(b) \subseteq \{d+1,\ldots,n-1\}}} q^{\deg_x(b)}. \]
    It only remains to show that for any $b \in B_n^{(1,2)}$, the condition $b \in B_{n}^{(1,2)}|_{(d+1)-\up}$ is equivalent to $\Asc(b) \subseteq \{d+1,\ldots,n-1\}$. 

    First suppose that $\Asc(b) \subseteq \{d+1,\ldots,n-1\}$, so then $m$ is not an ascent of $b$ for all $m \in \{1,\ldots, d\}$. Every modified Motzkin path must start with an up-step, so $\alpha_1 = \beta_1 = \gamma_1 = 0$. Then for $1$ to not be an ascent, $\beta_2 \neq 1$, so $\beta_2 = 0$. Furthermore, since $\beta_1 = \beta_2 = 0$, then for $1$ to not be an ascent, we must have $\alpha_1 \geq \alpha_2 + \gamma_2$. Since $\alpha_2, \gamma_2$ are both nonnegative, this implies that $\alpha_2 = \gamma_2 = 0$. This means that the second step must be an up-step with no $x$-degree contribution. Repeating this argument proves this direction.

    On the other hand, suppose that $b \in B_{n}^{(1,2)}|_{(d+1)-\up}$. Then it follows that $\alpha_m=\beta_m=\gamma_m=0$ for all $m \in \{1,\ldots,d+1\}$, so $\Asc(b) \subseteq \{d+1,\ldots,n-1\}$.
\end{proof}

\section{Characters for hook shapes}\label{sec:hooks}

In this section, we derive a formula for $\langle \SF(n,k,\ell) , s_\lambda \rangle$, when $\lambda$ is a hook shape $(d+1,1^{n-d-1}) \vdash n$ for $0 \leq d \leq n-1$. 

First, we prove the following proposition, which extends \cite[Proposition 5.7]{IraciNadeauVandenWyngaerd2024} for column shapes (at $d=0$) to hook shapes.

\begin{theorem}\label{thm:schur-coeff-hook}
    \[\langle \SF(n,k,\ell) , s_{(d+1,1^{n-d-1})}\rangle = \sum_{\substack{\sigma \in \SW(1^n,k,\ell), \\ \Split(\sigma) = \{d+1,\ldots, n-1\}}} q^{\sminv(\sigma)}= \sum_{\substack{b \in B_{n}^{(1,2)}, \\ \deg_\theta(b) = k,\\ \deg_\xi(b) = \ell,\\ \Asc(b) = \{d+1,\ldots,n-1\}}} q^{\deg_x(b)}.\]
\end{theorem}

\begin{proof}
    Since $\SF(n,k,\ell)$ is a symmetric function, its expansion into fundamental quasisymmetric functions $\SF(n,k,\ell) = \sum_{\alpha \vDash n} c_\alpha F_\alpha$ determines its Schur expansion $\SF(n,k,\ell) = \sum_{\alpha \vDash n} c_\alpha s_\alpha$ (see for example \cite{Gessel2019, GarsiaRemmel}).
    Note that the Schur function $s_\alpha$ is indexed by a composition, and it may be simplified to a Schur function indexed by a partition or $0$, using the ``slinky rule'' (see \cite{EggeLoehrWarrington}), which we describe briefly here.
    Draw a compositional Young diagram for $\alpha$ in French notation, and fix all left endpoints of each row, so they can never move. Then let each row fall down as far as possible, such that each row remains a ribbon (hence the term ``slinky''). If this forms a partition Young diagram $\lambda$, then $s_\alpha = \pm s_\lambda$, but if not, then $s_\alpha = 0$. (See Figure~\ref{fig:slinky-examples} for examples.)
    
    Since $ S = \Split(\sigma) = \{d+1,\ldots, n-1\}$, then $\alpha = (d+1,1^{n-d-1})$.
    We claim that the only composition $\alpha$ such that $s_\alpha = \pm s_{(d+1,1^{n-d-1})}$ is $\alpha = (d+1,1^{n-d-1})$. 
    Indexing the positions of the boxes in the compositional Young diagram with Cartesian coordinates starting from $(0,0)$ in the lower left corner. 
    When applying the slinky rule to any other $\alpha$ with a row longer than $1$ except for the bottom row would fall in such a way that a box would be either created in position $(1,1)$, and $\lambda$ would not be a hook shape, or would not fall all the way to $(1,1)$, and not be a partition Young diagram so the resulting Schur function would be $0$.

    Consider the coefficient of $s_{(d+1,1^{n-d-1})}$ in $\SF(n,k,\ell)$. We write:
    \begin{align*}
        \langle \SF(n,k,\ell) , s_{(d+1,1^{n-d-1})} \rangle &= \left\langle \sum_{\alpha \vDash n} c_\alpha F_\alpha , s_{(d+1,1^{n-d-1})} \right\rangle\\
        &= \left\langle \sum_{\alpha \vDash n} c_\alpha s_\alpha , s_{(d+1,1^{n-d-1})} \right\rangle\\
        &= c_{(d+1,1^{n-d-1})}.\\
    \end{align*}
    Let $[F_\alpha]X$ denote the coefficient of $F_\alpha$ in $X$. From the discussion on the slinky rule, we have that 
    \begin{align*}
       c_{(d+1,1^{n-d-1})} & = [F_{(d+1,1^{n-d-1})}] \SF(n,k,\ell)\\
        & = [F_{(d+1,1^{n-d-1})}] \SW_{z;q} (n,k,\ell)\\
        &= [Q_{\{d+1,\ldots,n-1\},n}]\sum_{\sigma \in \SW(1^n,k,\ell)} q^{\sminv(\sigma)} Q_{\Split(\sigma),n}\\
        &= \sum_{\substack{\sigma \in \SW(1^n,k,\ell), \\ \Split(\sigma) = \{d+1,\ldots, n-1\}}} q^{\sminv(\sigma)},
    \end{align*}
    where we used Theorem \ref{thm:INVW-3.1} and Proposition \ref{prop:INVW-quasisymmetric_expansion}.

    Finally, to convert into the formula in terms of $B_n^{(1,2)}$, we can restrict the bijection $\psi$ from Theorem~\ref{thm:bijection} to one on $\{b \in B_n^{(1,2)} \ | \  \deg_\theta(b) = k \text{ and } \deg_\xi(b) = \ell\}$ and $\SW(1^n,k,\ell)$ since it is $q,u,v$-weight preserving. Then the result follows from Proposition~\ref{prop:Split_on_basis}.
\end{proof}

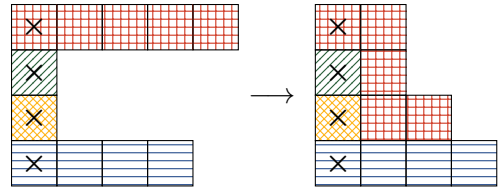
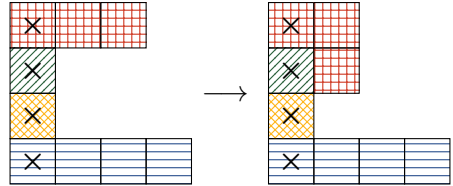
\begin{figure}[ht]
\centering
\begin{subfigure}[t]{0.5\textwidth}
\centering
\begin{tikzpicture}[scale=.6]
        \draw [pattern color=Lapis, pattern=horizontal lines] (0,0) rectangle (1,1) node[midway]{$\bm{\times}$};
        \draw [pattern color=Lapis, pattern=horizontal lines] (1,0) rectangle (2,1);
        \draw [pattern color=Lapis, pattern=horizontal lines] (2,0) rectangle (3,1);
        \draw [pattern color=Lapis, pattern=horizontal lines] (3,0) rectangle (4,1);
        \draw [pattern color=Gamboge, pattern=crosshatch] (0,1) rectangle (1,2) node[midway]{$\bm{\times}$};
        \draw [pattern color=Midori, pattern=north east lines] (0,2) rectangle (1,3) node[midway]{$\bm{\times}$};
        \draw [pattern color=Scarlet, pattern=grid] (0,3) rectangle (1,4) node[midway]{$\bm{\times}$};
        \draw [pattern color=Scarlet, pattern=grid] (1,3) rectangle (2,4);
        \draw [pattern color=Scarlet, pattern=grid] (2,3) rectangle (3,4);
        \draw [pattern color=Scarlet, pattern=grid] (3,3) rectangle (4,4);
        \draw [pattern color=Scarlet, pattern=grid] (4,3) rectangle (5,4);
\end{tikzpicture} \raisebox{2.88em}{$\longrightarrow$} \ 
\begin{tikzpicture}[scale=.6]
        \draw [pattern color=Lapis, pattern=horizontal lines] (0,0) rectangle (1,1) node[midway]{$\bm{\times}$};
        \draw [pattern color=Lapis, pattern=horizontal lines] (1,0) rectangle (2,1);
        \draw [pattern color=Lapis, pattern=horizontal lines] (2,0) rectangle (3,1);
        \draw [pattern color=Lapis, pattern=horizontal lines] (3,0) rectangle (4,1);
        \draw [pattern color=Gamboge, pattern=crosshatch] (0,1) rectangle (1,2) node[midway]{$\bm{\times}$};
        \draw [pattern color=Midori, pattern=north east lines] (0,2) rectangle (1,3) node[midway]{$\bm{\times}$};
        \draw [pattern color=Scarlet, pattern=grid] (0,3) rectangle (1,4) node[midway]{$\bm{\times}$};
        \draw [pattern color=Scarlet, pattern=grid] (1,3) rectangle (2,4);
        \draw [pattern color=Scarlet, pattern=grid] (1,2) rectangle (2,3);
        \draw [pattern color=Scarlet, pattern=grid] (1,1) rectangle (2,2);
        \draw [pattern color=Scarlet, pattern=grid] (2,1) rectangle (3,2);
\end{tikzpicture}
\caption{$s_{(5,1,1,4)} = \pm s_{(2,2,3,4)}$.}
\end{subfigure}
\vspace{1em}

\begin{subfigure}[t]{0.5\textwidth}
\centering
\begin{tikzpicture}[scale=.6]
        \draw [pattern color=Lapis, pattern=horizontal lines] (0,0) rectangle (1,1) node[midway]{$\bm{\times}$};
        \draw [pattern color=Lapis, pattern=horizontal lines] (1,0) rectangle (2,1);
        \draw [pattern color=Lapis, pattern=horizontal lines] (2,0) rectangle (3,1);
        \draw [pattern color=Lapis, pattern=horizontal lines] (3,0) rectangle (4,1);
        \draw [pattern color=Gamboge, pattern=crosshatch] (0,1) rectangle (1,2) node[midway]{$\bm{\times}$};
        \draw [pattern color=Midori, pattern=north east lines] (0,2) rectangle (1,3) node[midway]{$\bm{\times}$};
        \draw [pattern color=Scarlet, pattern=grid] (0,3) rectangle (1,4) node[midway]{$\bm{\times}$};
        \draw [pattern color=Scarlet, pattern=grid] (1,3) rectangle (2,4);
        \draw [pattern color=Scarlet, pattern=grid] (2,3) rectangle (3,4);
\end{tikzpicture} \raisebox{2.88em}{$\longrightarrow$} \ 
\begin{tikzpicture}[scale=.6]
        \draw [pattern color=Lapis, pattern=horizontal lines] (0,0) rectangle (1,1) node[midway]{$\bm{\times}$};
        \draw [pattern color=Lapis, pattern=horizontal lines] (1,0) rectangle (2,1);
        \draw [pattern color=Lapis, pattern=horizontal lines] (2,0) rectangle (3,1);
        \draw [pattern color=Lapis, pattern=horizontal lines] (3,0) rectangle (4,1);
        \draw [pattern color=Gamboge, pattern=crosshatch] (0,1) rectangle (1,2) node[midway]{$\bm{\times}$};
        \draw [pattern color=Midori, pattern=north east lines] (0,2) rectangle (1,3) node[midway]{$\bm{\times}$};
        \draw [pattern color=Scarlet, pattern=grid] (0,3) rectangle (1,4) node[midway]{$\bm{\times}$};
        \draw [pattern color=Scarlet, pattern=grid] (1,3) rectangle (2,4);
        \draw [pattern color=Scarlet, pattern=grid] (1,2) rectangle (2,3);
\end{tikzpicture}
\caption{$s_{(3,1,1,4)} = 0$.}
\end{subfigure}
    \caption{Examples of applying the slinky rule.}
    \label{fig:slinky-examples}
\end{figure}

We can characterize when $\Asc(b) = \{d+1,\ldots,n+1\}$. Via the weight-preserving bijection between $B_n^{(1,2)}$ and $\SW(1^n)$, this generalizes a characterization in the special case of $d=0$ \cite[Proposition 5.8]{IraciNadeauVandenWyngaerd2024}.

\begin{proposition}\label{prop:asc_characterization}
An element $b \in B_n^{(1,2)}$ satisfies $\Asc(b) = \{d+1,\ldots,n-1\}$ if and only if when $b$ is written as \[ b = \pm \prod_{m=1}^n x_m^{\alpha_m} \theta_m^{\beta_m} \xi_m^{\gamma_m},\]
for some $\alpha_m \in \Z_{\geq 0}$ and $\beta_m,\gamma_m \in \{0,1\}$, we have that for some $a \in \{d+1,\ldots,n\}$,
\begin{enumerate}[(a)]
    \item $\alpha_m = \beta_m = \gamma_m = 0$ for all $m \in \{1,\ldots, d+1\}$;
    \item $\beta_m = 0$ and $\alpha_{m-1} < \alpha_m + \gamma_m$ for all $m \in \{d+2,\ldots, a\}$;
    \item $\beta_a = 0$ and $\beta_{a+1} = 1$; and
    \item $\beta_m = 1$ and $\alpha_{m-1} \geq \alpha_m + \gamma_m$ for all $m \in \{a+2,\ldots,n\}$.
\end{enumerate}
\end{proposition}

Note that if $a=d+1$, then criterion (b) is skipped. If $a = n$, then criterion (c) is skipped. If $a = n-1$ or $a = n$, then criterion (d) is skipped.

\begin{proof}
    First, observe that if $b$ satisfies all of the given criteria, then $\Asc(b) = \{d+1,\ldots,n-1\}$, by applying the definition of an ascent of $b$.

    On the other hand, suppose that $\Asc(b) = \{d+1,\ldots,n-1\}$.
    From the characterization of $\Asc(b) \subseteq \{d+1,\ldots, n-1\}$ given in the proof of Corollary~\ref{cor:h_hook}, criterion (a) is satisfied for all $m \in \{1,\ldots,d+1\}$. 
    Next, since $d+1$ is an ascent, either $\beta_{d+2} = 0$ with $\alpha_{d+1} < \alpha_{d+2} + \gamma_{d+2}$ (satisfying criterion (b) for $m=d+2$), or $\beta_{d+2} = 1$ (satisfying criterion (c) for $a = d+1$).
    In the former case, we can continue satisfying criterion (b) for $m$ in an interval $\{d+1,\ldots,a\}$. Eventually, if we have not reached the end already, we will hit criterion (c) when some $\beta_{a+1} = 1$.
    At no point once we have reached any $\beta_i = 1$ will we be able to return to any $\beta_j =0$ for $j >i$, as an immediate consequence of the definition of an ascent of $b$.
    This implies that if we have not reached the end, since we continue to have ascents, we will continue to have $\beta_m =1$ and $\alpha_{m-1} \geq \alpha_m + \gamma_m$ for the remaining $m$, satisfying criterion (d). 
\end{proof}

As a consequence of \cite[Theorem 8.2]{D'Adderio-Romero}, the following explicit formula was given by Iraci, Nadeau, and Vanden Wyngaerd for $\lambda = (1^n)$, which indexes the sign character. 

\begin{theorem}[\!\!{\cite[Theorem 5.6]{IraciNadeauVandenWyngaerd2024}}]
    \[\langle \SF(n,k,\ell) , s_{(1^n)}\rangle = q^{\binom{n-k-\ell}{2}}\qbinom{n-1}{k+\ell}_q \qbinom{k+\ell}{k}_q.\]
\end{theorem}

Using the characterization given in Proposition~\ref{prop:asc_characterization}, we give the following generalization to hook shapes.

\begin{theorem}\label{thm:qbinomial} 
For $0 \leq d \leq n-1$,
    \[\langle \SF(n,k,\ell) , s_{(d+1,1^{n-d-1})}\rangle = q^{\binom{n-d-k-\ell}{2}}\qbinom{n-1-d}{\ell}_q\qbinom{n-1-k}{d}_q\qbinom{n-1-\ell}{k}_q.\]
\end{theorem}

To prove the theorem, we will make use of the following lemma, which is an immediate consequence of the $q$-Chu-Vandermonde identity.

\begin{lemma}\label{lem:Vandermonde}
    For any positive integer $n$ and nonnegative integers $k,\ell < n-d$,
    \[ q^{\binom{n-d-k-\ell}{2}}\qbinom{n-d-1}{\ell}_q  = \sum_{f=0}^\ell q^{\binom{n-d-k-f}{2} + \binom{\ell-f}{2}} \qbinom{n-d-1-k}{f}_q \qbinom{k}{\ell-f}_q .\]
\end{lemma}

\begin{proof}
Recall the $q$-Chu-Vandermonde identity (see for example \cite[Proposition 2.4]{IraciNadeauVandenWyngaerd2024} or \cite[Chapter 1, Solution to Exercise 100]{StanleyEC1}):
\[\qbinom{j}{a}_q = \sum_{i=0}^a q^{(r-i)(a-i)}\qbinom{r}{i}_q\qbinom{j-r}{a-i}_q.\]
Substituting $i \mapsto f$, $a \mapsto \ell$, $j \mapsto n-d-1$, and $r \mapsto n-d-1-k$ gives
\[\qbinom{n-d-1}{\ell}_q = \sum_{f=0}^\ell q^{(n-d-1-k-f)(\ell-f)}\qbinom{n-d-1-k}{f}_q\qbinom{k}{\ell-f}_q,\]
and the result follows from $\binom{n-d-k-f}{2} + \binom{\ell-f}{2} - \binom{n-d-k-\ell}{2} = (\ell-f)(n-d-1-k-f)$.
\end{proof}

\begin{proof}[Proof of Theorem~\ref{thm:qbinomial}]
    By Theorem~\ref{thm:schur-coeff-hook}, it remains to show that 
    \[\sum_{\substack{b \in B_{n}^{(1,2)}, \\ \deg_\theta(b) = k,\\ \deg_\xi(b) = \ell,\\ \Asc(b) = \{d+1,\ldots,n-1\}}} q^{\deg_x(b)} = q^{\binom{n-d-k-\ell}{2}}\qbinom{n-1-d}{\ell}_q\qbinom{n-1-k}{d}_q\qbinom{n-1-\ell}{k}_q.\]
    Recall the characterization of when $\Asc(b) = \{d+1,\ldots,n-1\}$ from Proposition~\ref{prop:asc_characterization}. Throughout the proof, we refer to Figure~\ref{fig:proof_picture} for a guiding example of what such a $b$ will look like.
    \begin{figure}[ht]
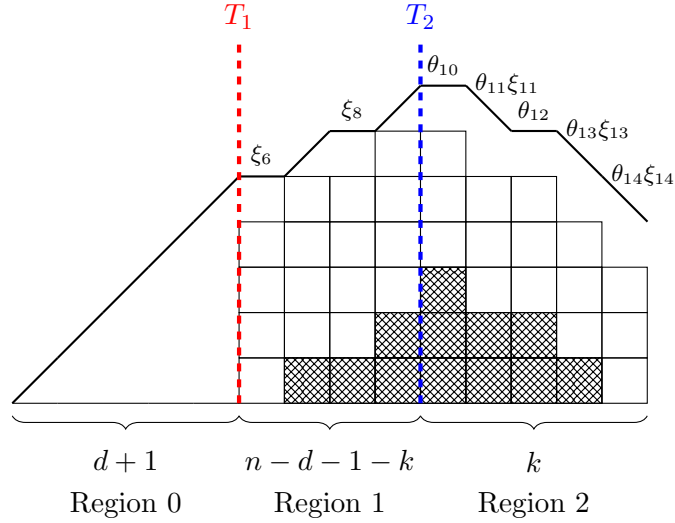

        \centering
        \motzkinpathbig{up,up,up,up,up,flat-xi,up,flat-xi,up,flat-theta,down,flat-theta,down,down}{0,0,0,0,0,4,5,5,6,6,5,5,4,3}{0,0,0,0,0,0,1,1,2,3,2,2,1,0}{}{5}{9}{8}
        \caption{A diagram showing the three regions for some $b$ with $\Asc(b) = \{d+1,\ldots,n-1\}$. The boxes shaded in crosshatch indicate where there is a minimal structure of boxes which must be filled in, and then above the minimal structure, boxes may be filled in, subject to certain increasing/decreasing conditions.}
        \label{fig:proof_picture}
    \end{figure}
    We record some facts on such a $b$. Region 0 will consist of only up-steps with no contribution to the $x$-degree. This corresponds to condition (a) of Proposition~\ref{prop:asc_characterization}. Region 0 cannot be empty, as $d + 1 \geq 1$. 
    
    At line $T_1$, we change to Region 1. This corresponds to condition (b) of Proposition~\ref{prop:asc_characterization}. The path in this region can be any sequence of up-steps or horizontal steps with decoration $\xi_i$. Region 1 may be empty. 
    Throughout Region 1, when there is a horizontal step with decoration $\xi_i$, the part of the staircase that is filled in must weakly increase, and when there is an up-step, it must strictly increase. 
    This implies that there exists some part of the staircase that must always be filled in, which we call the \textit{minimal structure for Region 1} and encode as a vector of column heights: $\MS(1) = (\MS(1)_1,\ldots, \MS(1)_{n-d-1-k})$. 
    We determine $\MS(1)$ recursively: for $1 \leq i \leq n-d-1-k$,
    \[\MS(1)_i = \begin{cases}
        \MS(1)_{i-1} + 1 &\text{ if $p_{d+1+i}$ is an up-step,}\\
        \MS(1)_{i-1} &\text{ if $p_{d+1+i}$ is a horizontal step with decoration $\xi_{d+1+i}$,}
    \end{cases}\]
    where we use the initial condition $\MS(1)_0 = 0$, although that is not considered part of the minimal structure. Recall that $p_{d+1+i}$ denotes the $(d+1+i)$-th step in the path, which is the $i$-th step in Region 1.

    Changing from Region 1 to Region 2 at line $T_2$ corresponds to condition (c) of Proposition~\ref{prop:asc_characterization}, and then the rest of Region 2 corresponds to condition (d). The path in this region can be any sequence of horizontal steps with decoration $\theta_i$ or down-steps with decoration $\theta_i\xi_i$. Region 2 may be empty. 
    Throughout Region 2, when there is a horizontal step with decoration $\theta_i$, the part of the staircase that is filled must weakly decrease, and when there is a down-step with decoration $\theta_i\xi_i$, it must strictly decrease. 
    Again, this implies that there exists some part of the staircase that must always be filled in, which we call the \textit{minimal structure for Region 2} and encode as a vector of column heights: $\MS(2) = (\MS(2)_1,\ldots, \MS(2)_{k})$. 
    We determine $\MS(2)$ recursively: for $k-1 \geq i \geq 1$,
    \[\MS(2)_i = \begin{cases}
        \MS(2)_{i+1} &\text{ if $p_{n-k+i}$ is a horizontal step with decoration $\xi_{n-k+i}$,}\\
        \MS(2)_{i+1} + 1 &\text{ if $p_{n-k+i}$ is a down-step with decoration $\theta_{n-k+i}\xi_{n-k+i}$,}
    \end{cases}\]
    with initial condition $\MS(2)_{k} = 0$. 
    Recall that $p_{n-k+i}$ denotes the $(n-k+i)$-th step in the path, which is the $i$-th step in Region 2.

    Now we are ready to prove the claimed $q$-enumeration. First, note that since $k,\ell,d,n$ are all fixed, the location of the lines $T_1$ and $T_2$, and thus the regions, are fixed. It only remains to count the possible paths and fillings in Regions 1 and 2. Let $f$ denote the number of horizontal steps with decoration $\xi_i$, all of which occur in Region 1. Note that there are $\binom{n-d-1-k}{f}$ ways to pick which of the steps in Region 1 will be horizontal steps with decoration $\xi_i$. 
    
    Consider $\MS(1)$: there is a classical staircase shape with heights $(1,2,3,\ldots,n-d-1-k-f)$ (corresponding to the up-steps) which is interleaved with $f$ additional components (corresponding to the horizontal steps with decoration $\xi_i$) which are equal in height to whatever immediately precedes them. So the contribution to the $q$-enumeration is $q^{\binom{n-d-k-f}{2}}$ by the classical staircase and $\qbinom{n-d-1-k}{f}_q$ by the additional interleaving components. See Figure~\ref{fig:region1} for an example.

    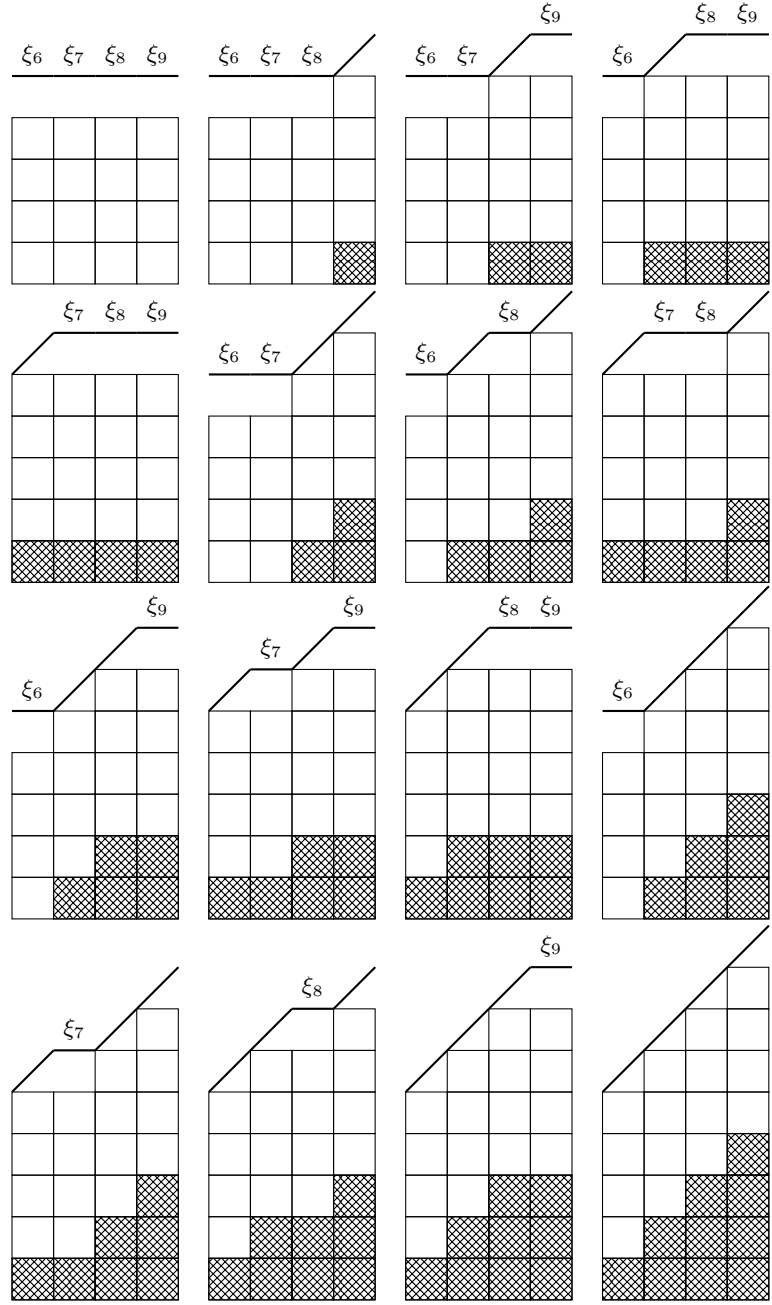
\begin{figure}[ht!]
        \centering
        \begin{tikzpicture}[scale =0.84]
                    \def\sep{3.1}
        \node[anchor=south west] at (0, 0) {
        \motzkinpathmiddle{up,up,up,up,up,flat-xi,flat-xi,flat-xi,flat-xi}{0,0,0,0,0,4,4,4,4}{0,0,0,0,0,0,0,0,0}{}{5}{9}{8}
        };
        \node[anchor=south west] at (\sep, 0) {
        \motzkinpathmiddle{up,up,up,up,up,flat-xi,flat-xi,flat-xi,up}{0,0,0,0,0,4,4,4,5}{0,0,0,0,0,0,0,0,1}{}{5}{9}{8}
        };
        \node[anchor=south west] at (2*\sep, 0) {
        \motzkinpathmiddle{up,up,up,up,up,flat-xi,flat-xi,up,flat-xi}{0,0,0,0,0,4,4,5,5}{0,0,0,0,0,0,0,1,1}{}{5}{9}{8}
        };
        \node[anchor=south west] at (3*\sep, 0) {
        \motzkinpathmiddle{up,up,up,up,up,flat-xi,up,flat-xi,flat-xi}{0,0,0,0,0,4,5,5,5}{0,0,0,0,0,0,1,1,1}{}{5}{9}{8}
        };
        \node[anchor=south west] at (0, -4.7) {
        \motzkinpathmiddle{up,up,up,up,up,up,flat-xi,flat-xi,flat-xi}{0,0,0,0,0,5,5,5,5}{0,0,0,0,0,1,1,1,1}{}{5}{9}{8}
        };
        \node[anchor=south west] at (\sep, -4.7) {
        \motzkinpathmiddle{up,up,up,up,up,flat-xi,flat-xi,up,up}{0,0,0,0,0,4,4,5,6}{0,0,0,0,0,0,0,1,2}{}{5}{9}{8}
        };
        \node[anchor=south west] at (2*\sep, -4.7) {
        \motzkinpathmiddle{up,up,up,up,up,flat-xi,up,flat-xi,up}{0,0,0,0,0,4,5,5,6}{0,0,0,0,0,0,1,1,2}{}{5}{9}{8}
        };
        \node[anchor=south west] at (3*\sep, -4.7) {
        \motzkinpathmiddle{up,up,up,up,up,up,flat-xi,flat-xi,up}{0,0,0,0,0,5,5,5,6}{0,0,0,0,0,1,1,1,2}{}{5}{9}{8}
        };
        \node[anchor=south west] at (0, -10) {
        \motzkinpathmiddle{up,up,up,up,up,flat-xi,up,up,flat-xi}{0,0,0,0,0,4,5,6,6}{0,0,0,0,0,0,1,2,2}{}{5}{9}{8}
        };
        \node[anchor=south west] at (\sep, -10) {
        \motzkinpathmiddle{up,up,up,up,up,up,flat-xi,up,flat-xi}{0,0,0,0,0,5,5,6,6}{0,0,0,0,0,1,1,2,2}{}{5}{9}{8}
        };
        \node[anchor=south west] at (2*\sep, -10) {
        \motzkinpathmiddle{up,up,up,up,up,up,up,flat-xi,flat-xi}{0,0,0,0,0,5,6,6,6}{0,0,0,0,0,1,2,2,2}{}{5}{9}{8}
        };
        \node[anchor=south west] at (3*\sep, -10) {
        \motzkinpathmiddle{up,up,up,up,up,flat-xi,up,up,up}{0,0,0,0,0,4,5,6,7}{0,0,0,0,0,0,1,2,3}{}{5}{9}{8}
        };
        \node[anchor=south west] at (0, -16) {
        \motzkinpathmiddle{up,up,up,up,up,up,flat-xi,up,up}{0,0,0,0,0,5,5,6,7}{0,0,0,0,0,1,1,2,3}{}{5}{9}{8}
        };
        \node[anchor=south west] at (\sep, -16) {
        \motzkinpathmiddle{up,up,up,up,up,up,up,flat-xi,up}{0,0,0,0,0,5,6,6,7}{0,0,0,0,0,1,2,2,3}{}{5}{9}{8}
        };
        \node[anchor=south west] at (2*\sep, -16) {
        \motzkinpathmiddle{up,up,up,up,up,up,up,up,flat-xi}{0,0,0,0,0,5,6,7,7}{0,0,0,0,0,1,2,3,3}{}{5}{9}{8}
        };
        \node[anchor=south west] at (3*\sep, -16) {
        \motzkinpathmiddle{up,up,up,up,up,up,up,up,up}{0,0,0,0,0,5,6,7,8}{0,0,0,0,0,1,2,3,4}{}{5}{9}{8}
        };
                \end{tikzpicture}
        \caption{All possible paths for Region 1, when $d+1 = 5$ and $n-d-1-k = 4$, along with each path's minimal structure shaded in crosshatch. The contribution of $f=4$ is $q^{\binom{1}{2}}\qbinom{4}{4}_q = 1$, of $f=3$ is $q^{\binom{2}{2}}\qbinom{4}{3}_q = q+q^2+q^3+q^4$, of $f=2$ is $q^{\binom{3}{2}}\qbinom{4}{2}_q = q^3+q^4+2q^5+q^6+q^7$, of $f=1$ is $q^{\binom{4}{2}}\qbinom{4}{1}_q = q^6+q^7+q^8+q^9$, and of $f=0$ is $q^{\binom{5}{2}}\qbinom{4}{0}_q = q^{10}$.}
        \label{fig:region1}
    \end{figure}

    Similarly, consider $\MS(2)$: there is a reversed classical staircase shape with heights $(\ell-f-1,\ldots,2,1,0)$ (corresponding to the down-steps with decoration $\theta_i\xi_i$) which is interleaved with $\ell-f$ additional components (corresponding to the horizontal steps with decoration $\theta_i$) which are equal in height to whatever immediately precedes them (or height $\ell-f$ if nothing in Region 2 precedes it). So the contribution to the $q$-enumeration is $q^{\binom{\ell-f}{2}}$ by the classical staircase and $\qbinom{k}{\ell-f}_q$ by the additional interleaving components.
    
    By Lemma~\ref{lem:Vandermonde}, by summing over all possible values for $f$, we get that the contribution of the minimal structures $\MS(1)$ and $\MS(2)$ to the $q$-enumeration is $q^{\binom{n-d-k-\ell}{2}}\qbinom{n-d-1}{\ell}_q$.

    All that remains is to consider how we can fill in the region above the minimal structure. First, observe that by construction, the boxes for which we have choice to fill in are equinumerous within each column in Region 1 (there are $d$ boxes) and within Region 2 (there are $n-k-\ell-1$). There is one more choice for filling a column than the number of boxes, since it is possible to leave all empty. In Region 1, the criteria that:
    \begin{itemize}
        \item when there is a horizontal step with decoration $\xi_i$, the part of the staircase that is filled must weakly increase, and when there is an up-step, it must strictly increase
    \end{itemize}
    upon deletion of the minimal structure $\MS(1)$, and letting the remaining boxes fall by gravity becomes: 
    \begin{itemize}
        \item the part of the staircase that is filled must weakly increase.
    \end{itemize} See Figure~\ref{fig:gravity}.
    \begin{figure}
        \centering
        \begin{tikzpicture}
        \node[anchor=south west] at (0, 0) {
        \motzkinpathmiddle{up,up,up,up,up,flat-xi,up,flat-xi,up}{0,0,0,0,0,4,5,5,6}{0,0,0,0,0,0,1,1,2}{}{5}{9}{8} };
        \node[anchor=south west] at (2.8, 2) {$\longrightarrow$};
        \node[anchor=south west] at (2.6, 0) {
        \motzkinpathmiddle{up,up,up,up,up}{0,0,0,0,0,4,4,4,4}{0,0,0,0,0,0,0,0,0}{}{5}{9}{8} };
        \end{tikzpicture}
        \caption{A demonstration of deleting the minimal structure from the maximal staircase, and letting the remaining boxes fall by gravity. This will always result in a rectangular shape.}
        \label{fig:gravity}
    \end{figure}
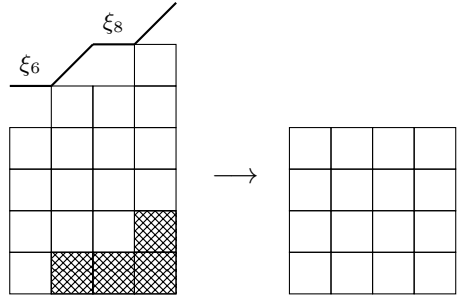
    Then this is counted by the number of length $n-d-1-k$ sequences on $d+1$ letters, which is $\multiset{d+1}{n-d-1-k} = \binom{n-1-k}{d}$.\footnote{$\multiset{a}{b}$ denotes the multiset coefficient (see for example \cite[Chapter 1.2]{StanleyEC1}).} Keeping track of the effect of the fillings on the $q$-weights gives $\qbinom{n-1-k}{d}_q$.

    Following similar analysis for Region 2, where we delete the minimal structure $\MS(2)$, we find what we need to count is the number of length $k$ sequences on $n-k-\ell$ letters, which is $\multiset{n-k-\ell}{k} = \binom{n-1-\ell}{k}$. Keeping track of the effect of the fillings on the $q$-weights gives $\qbinom{n-1-\ell}{k}_q$.

    Finally, multiplying $q^{\binom{n-d-k-\ell}{2}}\qbinom{n-d-1}{\ell}_q $ by $\qbinom{n-1-k}{d}_q$ and $\qbinom{n-1-\ell}{k}_q$ proves the theorem.
\end{proof}

\section{The conjectural basis in type B}\label{sec:typeB}

In this section, we describe analogous results and conjectures where the Weyl group of type $\mathfrak{B}_n$ replaces the symmetric group $\mathfrak{S}_n$. We denote the ring in question by $R_{\mathfrak{B}_n}^{(1,2)}$. We describe the Kim--Rhoades basis for $R_{\mathfrak{B}_n}^{(0,2)}$ and Sagan--Swanson's conjectural super-Artin basis for $R_{\mathfrak{B}_n}^{(1,1)}$.

Define the set of \textit{(type $B$) modified Motzkin paths} of length $n$, $\Pi(n)_{\geq 0}$,\footnote{Note the difference in notation between this and in Section 2. The difference between the definitions is that paths either must not go below $y=0$ in $\Pi(n)_{\geq 0}$, or $y=1$ after the first step in $\Pi(n)_{> 0}$.} to be the set of all paths $\pi = (p_1,\ldots,p_n)$ in $\Z^2$ such that each step $p_i$ is one of: 
\begin{enumerate}[(a)]
    \item an up-step $(1,1)$,
    \item a horizontal step $(1,0)$ with decoration $\theta_i$,
    \item a horizontal step $(1,0)$ with decoration $\xi_i$,
    \item or a down-step $(1,-1)$ with decoration $\theta_i\xi_i$,
\end{enumerate}   
where the path never goes below the horizontal line $y=0$. The weight of $\pi \in \Pi(n)_{\geq 0}$ is defined in the same way as in Section~\ref{sec:background}.

\begin{definition}[\!\!\cite{KimRhoades2022}]
    The \textit{type $B$ Kim--Rhoades basis} $B_{\mathfrak{B}_n}^{(0,2)}$ is the set of all weights of the modified Motzkin paths $\pi \in \Pi(n)_{\geq0}$, that is,
    \[ B_{\mathfrak{B}_n}^{(0,2)} := \{ \wt(\pi)\, | \, \pi \in \Pi(n)_{\geq0} \}.\]
\end{definition} 

We are justified in calling it a basis because of the following result, which is a special case of a more general, type-independent result of Kim and Rhoades.

\begin{theorem}[\!\!\cite{KimRhoades2022}]
    The type $B$ Kim--Rhoades basis $B_{\mathfrak{B}_n}^{(0,2)}$ is a basis for $R_{\mathfrak{B}_n}^{(0,2)}$.
\end{theorem}

Next, we recall the conjectural type $B$ super-Artin basis, due to Sagan and Swanson. Recall that $\chi(P)$ is $1$ if the proposition $P$ is true, and $0$ if the proposition $P$ is false. For any $T \subseteq \{1,\ldots,n\}$, define the \textit{$\beta$-sequence} $\beta(T) = (\beta_1(T), \ldots, \beta_n(T))$ recursively by the initial condition $\beta_1(T) = \chi(1 \not\in T)$ and for $2 \leq  i \leq n$, 
\[
    \beta_i(T) =\beta_{i-1}(T) + \chi(i \not\in T) + \chi(i-1 \not\in T).
\]
\begin{definition}[\!\!\cite{SaganSwanson2024}]
    The type $B$ super-Artin set is \[B_{\mathfrak{B}_n}^{(1,1)} := \{x^\beta \theta_T \, | \, T \subseteq \{1,\ldots,n\} \text{ and } \beta \leq \beta(T) \text{ componentwise}\}.\]
\end{definition}

\begin{conjecture}[\!\!\cite{SaganSwanson2024}]
    The type $B$ super-Artin set $B_{\mathfrak{B}_n}^{(1,1)}$ is a basis for $R_{\mathfrak{B}_n}^{(1,1)}$.
\end{conjecture}

We generalize the $\beta$-sequence as follows. For any $T,S \subseteq \{1,\ldots, n\}$, define the \textit{generalized $\beta$-sequence} by $\beta(T,S) = (\beta_1(T,S), \ldots, \beta_n(T,S))$ recursively by the initial condition $\beta_1(T,S) = -1 + \chi(1 \not\in T) + \chi(1 \not\in S)$ and for $2 \leq  i \leq n$, 
\[
    \beta_i(T,S) = \beta_{i-1}(T,S) -2 +\chi(i \not\in T) + \chi(i-1 \not\in T) + \chi(i \not\in S) + \chi(i-1 \not\in S).
\]

\begin{definition}\label{def:B_{B_n}^{(1,2)}}
     We let \[B_{\mathfrak{B}_n}^{(1,2)} := \{x^\beta \theta_T \xi_S \, | \, \theta_T \xi_S \in B_{\mathfrak{B}_n}^{(0,2)} \text{ and } 0 \leq \beta_i \leq \beta_i(T,S) \text{ for all } i \in \{1,\ldots,n\}\}.\]
\end{definition}

For example, consider $B_{\mathfrak{B}_6}^{(1,2)}$. To construct this set, we start with $B_{\mathfrak{B}_6}^{(0,2)}$. One such element in $B_{\mathfrak{B}_6}^{(0,2)}$ is $\theta_{\{3,4\}}\xi_{\{3,6\}} = \theta_3\theta_4\xi_3\xi_6$. We compute the generalized $\beta$-sequence $\beta(\{3,4\},\{3,6\}) = (1,3,3,2,3,4)$. Then to construct the elements in $B_{\mathfrak{B}_6}^{(1,2)}$ associated to $\theta_3\theta_4\xi_3\xi_6$, multiply by $x^\beta$, where $\beta$ satisfies $0 \leq \beta_i \leq \beta_i(\{3,4\},\{3,6\})$. See Figure~\ref{fig:type_B_example} for one of the $2\cdot4\cdot4\cdot3\cdot4\cdot5$, elements associated to the modified Motzkin path, and observe that in type $B$, the \textit{maximal staircase} given by $\beta(T,S)$ in general does not fit underneath the modified Motzkin path.

\begin{figure}[ht]
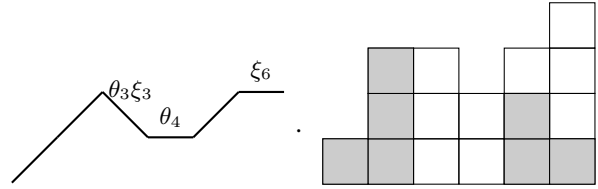

    \centering
\motzkin{up,up,down,flat-theta,up,flat-xi}{}$\cdot$ \raisebox{0.8em}{\motzkinpath{skip,skip,skip,skip,skip,skip}{1,3,3,2,3,4}{1,3,0,0,2,1}{}}
    \caption{An element $x_1x_2^3x_5^2x_6\theta_3\theta_4\xi_3\xi_6$ in $B_{\mathfrak{B}_6}^{(1,2)}$, represented by a type $B$ modified Motzkin path and partially filled-in staircase.}
    \label{fig:type_B_example}
\end{figure}

\begin{conjecture}\label{conj:R_{B_n}^{(1,2)}}
    The set $B_{\mathfrak{B}_n}^{(1,2)}$ is a basis for $R_{\mathfrak{B}_n}^{(1,2)}$.
\end{conjecture}

Define
\[ \stair_q^B(\pi) := \prod_{k \in \beta(T(\pi),S(\pi))} [k+1]_q,\]
where $T(\pi)$ and $S(\pi)$ are determined by which elements in $\{1,\ldots,n\}$ appear as indices for $\theta_i$ and $\xi_i$ respectively in the weight of the modified Motzkin path $\pi$. As before, $\deg_\theta(\pi) = |T(\pi)|$ and $\deg_\xi(\pi) = |S(\pi)|$.
The following result follows from a similar argument as that used to show Proposition~\ref{prop:Hilb(R_n^{(1,2)})}.

\begin{proposition}\label{prop:Hilb(R_{B_n}^{(1,2)})}
    Assuming Conjecture~\ref{conj:R_{B_n}^{(1,2)}}, it follows that the Hilbert series of $R_{\mathfrak{B}_n}^{(1,2)}$ is 
    \[ \sum_{\pi \in \Pi_{\geq 0}(n)} u^{\deg_\theta(\pi)} v^{\deg_\xi(\pi)} \stair_q^B(\pi).\]
\end{proposition}

Based on computations for $n \leq 5$, we make the following conjecture on the dimension of $R_{\mathfrak{B}_n}^{(1,2)}$, which to our knowledge, has not previously appeared. 
\begin{conjecture}\label{conj:dimR_{B_n}^{(1,2)}}
    The dimension of $R_{\mathfrak{B}_n}^{(1,2)}$ is $4^{n}n!$.
\end{conjecture}

\begin{remark}
Note that the conjectural dimension of $R_n^{(1,2)}$ and $R_{\mathfrak{B}_n}^{(1,2)}$ are both given by $2^r|W|$, where $r$ is the rank of $W$. 
%This is also true for dihedral groups, as we will show in a separate article.
It would be interesting to know for which finite Coxeter groups this holds, and for which it fails. Type-independent work in the case of one set of bosonic and one set of fermionic variables has been advanced by Swanson and Wallach \cite{SwansonWallach1, SwansonWallach2}.
\end{remark}

\begin{theorem}\label{thm:typeB_enumeration}
    The cardinality of $B_{\mathfrak{B}_n}^{(1,2)}$ is $4^n n!$.
\end{theorem}

\begin{proof}
    Let $P_B(n,r)$ denote the subset of elements of $B_{\mathfrak{B}_n}^{(1,2)}$ which have maximal staircase ending at height $r$.\footnote{Unlike in the case of $B_n^{(1,2)}$, as in the proof of Theorem~\ref{thm:cardinality}, here the height of the modified Motzkin path does not in general correspond to the height of the maximal staircase.}

    Since in type $B$, both the current step and the previous step in $\pi$ affect the shape of the maximal staircase at the current position, it will be convenient to introduce the following notation. Let $P_E(n,r)$ denote the subset of $P_B(n,r)$ whose current step (that is, step $n$) of the corresponding path is a horizontal step. Let $P_U(n,r)$ denote the subset of $P_B(n,r)$ whose current step is an up-step. Let $P_D(n,r)$ denote the subset of $P_B(n,r)$ whose current step is a down-step. It follows that $P_B(n,r) = P_E(n,r) \sqcup P_U(n,r) \sqcup P_D(n,r)$. Also denote $p_B(n,r) := |P_B(n,r)|$, $p_E(n,r) := |P_E(n,r)|$, $p_U(n,r) := |P_U(n,r)|$, and $p_D(n,r) := |P_D(n,r)|$. Thus $p_B(n,r) = p_E(n,r) + p_U(n,r) + p_D(n,r)$.
    
    Consider $\pi \in P_E(n,r)$: since step $n$ is a horizontal step, it does not affect the height of the maximal staircase. But whether step $n-1$ in $\pi$ is either a horizontal step, an up-step, or a down-step does affect the height of the maximal staircase: specifically, we get $p_E(n-1,r) + p_U(n-1,r-1)+p_D(n-1,r+1)$ as the possibilities, and then we multiply by $2(r+1)$ to account for the two choices of weight on $n$, either $\theta_n$ or $\xi_n$, along with the $r+1$ choices for the exponent of $x_n$. This implies that
    \[ p_E(n,r) = 2(r+1)\left(p_E(n-1,r) + p_U(n-1,r-1)+p_D(n-1,r+1)\right).\]
    Similarly, we find that 
    \[ p_U(n,r) = (r+1)\left(p_E(n-1,r-1) + p_U(n-1,r-2)+p_D(n-1,r)\right),\]
    and
    \[ p_D(n,r) = (r+1)\left(p_E(n-1,r+1) + p_U(n-1,r)+p_D(n-1,r+2)\right).\]
    The initial conditions are given by:
    $p_E(0,r)=\delta_{0,r}$, $p_U(0,r)=0$, and $p_D(0,r)=0$.
    It follows that $p_E(n,2r+1)=0$, $p_U(n,2r)=0$, and $p_D(n,2r)=0$ for any nonnegative integers $r$.
    
    Next, we claim that for $n \geq 1$,
    \[p_E(n,2r) = (n-1)!2^n\binom{n-1}{r}(2r+1),\]
    \[p_U(n,2r+1) = (n-1)!2^n\binom{n-1}{r}(r+1),\]
    \[p_D(n,2r+1) = (n-1)!2^n\binom{n-1}{r}(n-r-1).\]
    We show these by induction on $n$. For the base case of $n=1$, 
    \[p_E(1,2r)= \begin{cases}
        2&\text{ if } r=0,\\
        0&\text{ otherwise, }
    \end{cases}\]
    since there are two elements in $B_{\mathfrak{B}_1}^{(1,2)}$ with maximal staircase of height $0$: $\theta_1$ and $\xi_1$, and none with maximal staircase of height greater than $0$. Similarly, we check that
    \[p_U(1,2r+1)= \begin{cases}
        2&\text{ if } r=0,\\
        0&\text{ otherwise, }
    \end{cases}\]
    since there are two elements in $B_{\mathfrak{B}_1}^{(1,2)}$, $1$ and $x_1$, each with maximal staircase of height $1$, and none with maximal staircase of height greater than $1$. 
    Finally, $p_U(1,2r+1)= 0$, as there are no paths which begin with a down-step.

    Now assume that the claim holds for fixed $n \geq 1$. Using the previously established recurrence relations, we have that
    \begin{align*}
        &p_E (n+1,2r)\\ 
        &\quad= 2(2r+1)\left[(n-1)!2^n\left(\binom{n-1}{r}(2r+1) + \binom{n-1}{r-1}r + \binom{n-1}{r}(n-r-1)\right)\right]\\
        &\quad= (2r+1)(n-1)!2^{n+1}\left[n\binom{n-1}{r} + r\left(\binom{n-1}{r}+\binom{n-1}{r-1}\right)\right]\\
        &\quad= n!2^{n+1}\binom{n}{r}(2r+1).
    \end{align*}
    Similarly, we have 
    \begin{align*}
        &p_U (n+1,2r+1)\\ 
        &\quad= 2(r+1)\left[(n-1)!2^n\left(\binom{n-1}{r}(2r+1) + \binom{n-1}{r-1}r + \binom{n-1}{r}(n-r-1)\right)\right]\\
        &\quad= (r+1)(n-1)!2^{n+1}\left[n\binom{n-1}{r} + r\left(\binom{n-1}{r}+\binom{n-1}{r-1}\right)\right]\\
        &\quad= n!2^{n+1}\binom{n}{r}(r+1),
    \end{align*}
    and 
    \begin{align*}
        &p_D (n+1,2r+1)\\
        &\quad= 2(r+1)\\
        &\quad\quad\cdot\left[(n-1)!2^n\left(\binom{n-1}{r+1}(2r+3) + \binom{n-1}{r}(r+1) + \binom{n-1}{r+1}(n-r-2)\right)\right]\\
        &\quad= (r+1)(n-1)!2^{n+1}\left[n\binom{n-1}{r+1} + (r+1)\left(\binom{n-1}{r+1}+\binom{n-1}{r}\right)\right]\\
        &\quad= n!2^{n+1}\binom{n}{r+1}(r+1)\\
        &\quad= n!2^{n+1}\binom{n}{r}(n-r).
    \end{align*}
    This completes our induction.

    Considering $p_B(n,s)$, depending on the parity of $s$, we get that
    \[p_B(n,2r) = p_E (n,2r) = (n-1)!2^n\binom{n-1}{r}(2r+1),\]
    and 
    \[p_B(n,2r+1) = p_U(n,2r+1) + p_D(n,2r+1) = n!2^n\binom{n-1}{r}.\]

    Putting everything together, we have that
    \begin{align*}
        \sum_{s=0}^{2n+1}p_B(n,s)&= \sum_{r=0}^{n}p_B(n,2r) + \sum_{r=0}^{n}p_B(n,2r+1)\\
        &=\sum_{r=0}^{n}(n-1)!2^n\binom{n-1}{r}(2r+1) + \sum_{r=0}^{n}n!2^n\binom{n-1}{r}\\
        &= (n-1)!2^n\sum_{r=0}^{n}\binom{n-1}{r}(2r+1 +n)\\
        &= 4^n n!,
    \end{align*}
    where the last line follows from the identity $n2^n = \sum_{r=0}^{n}\binom{n-1}{r}(2r+1 +n)$, which can be shown by induction.
\end{proof}

%%%%%%%%%%%%%%%%%%%%%%%%%%%%%%%%%%%%%%%%%%%%%%%%%%%%%%%%%%%%%%%%%%%%%

\section*{Acknowledgements}

The author would like to thank Fran\c{c}ois Bergeron, Sylvie Corteel, Nicolle Gonz\'alez, Sean Griffin, Mark Haiman, Brendon Rhoades, Josh Swanson, and Mike Zabrocki for invaluable conversations and/or feedback on a draft of this paper, and the anonymous referees for feedback which improved the exposition. The author was supported by the National Science Foundation Graduate Research Fellowship DGE-2146752.

\bibliographystyle{amsplain}
\bibliography{12paper/sample_2025}

\appendix
\section{The bases \texorpdfstring{$B_n^{(1,2)}$ and $\SW(1^n)$ for $n = 1,2,3$}{Bn12 for n=1,2,3}}\label{app:A}

This appendix consists of tables which show the bijection between segmented permutations $\sigma \in \SW(1^n)$ and basis elements $b \in B_n^{(1,2)}$. In each segmented permutation, the thin indices are underlined.

\begin{table}[ht]
\resizebox{\columnwidth}{!}{%
\centering
\begin{tabular}{|c|c|c|c|c|c|c|c|}
\hline
\makecell{\textbf{Segmented} \\ \textbf{Permutation $\sigma$}} & \makecell{\textbf{Basis} \\ \textbf{Element $b$}} & \makecell{\textbf{Path} \\ \textbf{Model}} & \makecell{$k(\sigma) = $\\$ \deg_\theta(b)$} & \makecell{$\ell(\sigma) = $ \\ $ \deg_\xi(b)$} & \makecell{$\sminv(\sigma) = $ \\ $\deg_x(b)$} & \textbf{$\sdinv(\sigma)$}\tablefootnote{We do not define $\sdinv$ in this paper because we do not otherwise use it. It is equidistributed with $\sminv$ and is defined in \cite[Section 4.3]{IraciNadeauVandenWyngaerd2024}.}  & \makecell{$\Split(\sigma) = $ \\ $\Asc(b)$} \\
\hline
$1$ & $1$ & \motzkinpathnolabel{up}{0}{0}{} & 0 & 0 & 0 & 0 & $\varnothing$ \\
\hline
\end{tabular}
}
\caption{Conversion between $B_1^{(1,2)}$ and $\SW(1)$ with statistics.}
\label{tab:B_1^{(1,2)}}
\end{table}

\begin{table}[ht]
\resizebox{\columnwidth}{!}{%
\centering
\begin{tabular}{|c|c|c|c|c|c|c|c|}
\hline
\makecell{\textbf{Segmented} \\ \textbf{Permutation $\sigma$}} & \makecell{\textbf{Basis} \\ \textbf{Element $b$}} & \makecell{\textbf{Path} \\ \textbf{Model}} & \makecell{$k(\sigma) = $\\$ \deg_\theta(b)$} & \makecell{$\ell(\sigma) = $ \\ $ \deg_\xi(b)$} & \makecell{$\sminv(\sigma) = $ \\ $\deg_x(b)$} & \textbf{$\sdinv(\sigma)$} & \makecell{$\Split(\sigma) = $ \\ $\Asc(b)$} \\
\hline
$1\underline{2}$ & $\theta_2$ & \motzkinpathnolabel{up,flat-theta}{0,0}{0,0}{} & 1 & 0 & 0 & 0 & $\{1\}$ \\
\hline
$21$ & $\xi_2$ & \motzkinpathnolabel{up,flat-xi}{0,0}{0,0}{} & 0 & 1 & 0 & 0 & $\{1\}$ \\
\hline
$1|2$ & $1$ & \motzkinpathnolabel{up,up}{0,1}{0,0}{} & 0 & 0 & 0 & 0 & $\varnothing$ \\
\hline
$2|1$ & $x_2$ & \motzkinpathnolabel{up,up}{0,1}{0,1}{} & 0 & 0 & 1 & 1 & $\{1\}$ \\
\hline
\end{tabular}
}
\caption{Conversion between $B_2^{(1,2)}$ and $\SW(1^2)$ with statistics.}
\label{tab:B_2^{(1,2)}}
\end{table}

\begin{table}[ht]
\resizebox{\columnwidth}{!}{%
\centering
\begin{tabular}{|c|c|c|c|c|c|c|c|}
\hline
\makecell{\textbf{Segmented} \\ \textbf{Permutation $\sigma$}} & \makecell{\textbf{Basis} \\ \textbf{Element $b$}} & \makecell{\textbf{Path} \\ \textbf{Model}} & \makecell{$k(\sigma) = $\\$ \deg_\theta(b)$} & \makecell{$\ell(\sigma) = $ \\ $ \deg_\xi(b)$} & \makecell{$\sminv(\sigma) = $ \\ $\deg_x(b)$} & \textbf{$\sdinv(\sigma)$} & \makecell{$\Split(\sigma) = $ \\ $\Asc(b)$} \\
\hline
$1\underline{2}\underline{3}$ & $\theta_2\theta_3$ & \motzkinpathnolabel{up,flat-theta,flat-theta}{0,0,0}{0,0,0}{} & 2 & 0 & 0 & 0 & $\{1,2\}$ \\
\hline
$1\underline{3}2$ & $\theta_3\xi_3$ & \motzkinpathnolabel{up,up,down}{0,1,0}{0,0,0}{} & 1 & 1 & 0 & 0 & $\{2\}$ \\
\hline
$21\underline{3}$ & $\xi_2\theta_3$ & \motzkinpathnolabel{up,flat-xi,flat-theta}{0,0,0}{0,0,0}{} & 1 & 1 & 0 & 0 & $\{1,2\}$ \\
\hline
$2\underline{3}1$ & $x_2\theta_3\xi_3$ & \motzkinpathnolabel{up,up,down}{0,1,0}{0,1,0}{} & 1 & 1 & 1 & 1 & $\{1,2\}$ \\
\hline
$31\underline{2}$ & $\theta_2\xi_3$ & \motzkinpathnolabel{up,flat-theta,flat-xi}{0,0,0}{0,0,0}{} & 1 & 1 & 0 & 0 & $\{1\}$ \\
\hline
$321$ & $\xi_2\xi_3$ & \motzkinpathnolabel{up,flat-xi,flat-xi}{0,0,0}{0,0,0}{} & 0 & 2 & 0 & 0 & $\{1,2\}$ \\
\hline
$1|2\underline{3}$ & $\theta_3$ & \motzkinpathnolabel{up,up,flat-theta}{0,1,1}{0,0,0}{} & 1 & 0 & 0 & 1 & $\{2\}$ \\
\hline
$1|32$ & $\xi_3$ & \motzkinpathnolabel{up,up,flat-xi}{0,1,1}{0,0,0}{} & 0 & 1 & 0 & 0 & $\{2\}$ \\
\hline
$2|1\underline{3}$ & $x_2\theta_3$ & \motzkinpathnolabel{up,up,flat-theta}{0,1,1}{0,1,0}{} & 1 & 0 & 1 & 2 & $\{1,2\}$ \\
\hline
$2|31$ & $x_2\xi_3$ & \motzkinpathnolabel{up,up,flat-xi}{0,1,1}{0,1,0}{} & 0 & 1 & 1 & 1 & $\{1\}$ \\
\hline
$3|1\underline{2}$ & $x_3\theta_2$ & \motzkinpathnolabel{up,flat-theta,up}{0,0,1}{0,0,1}{} & 1 & 0 & 1 & 1 & $\{1\}$ \\
\hline
$3|21$ & $x_3\xi_2$ & \motzkinpathnolabel{up,flat-xi,up}{0,0,1}{0,0,1}{} & 0 & 1 & 1 & 1 & $\{1,2\}$ \\
\hline
\end{tabular}
}
\caption{Conversion between $B_3^{(1,2)}$ and $\SW(1^3)$ with statistics (first half).}
\label{tab:first_half}
\end{table}

\begin{table}[ht]
\resizebox{\columnwidth}{!}{%
\centering
\begin{tabular}{|c|c|c|c|c|c|c|c|}
\hline
\makecell{\textbf{Segmented} \\ \textbf{Permutation $\sigma$}} & \makecell{\textbf{Basis} \\ \textbf{Element $b$}} & \makecell{\textbf{Path} \\ \textbf{Model}} & \makecell{$k(\sigma) = $\\$ \deg_\theta(b)$} & \makecell{$\ell(\sigma) = $ \\ $ \deg_\xi(b)$} & \makecell{$\sminv(\sigma) = $ \\ $\deg_x(b)$} & \textbf{$\sdinv(\sigma)$} & \makecell{$\Split(\sigma) = $ \\ $\Asc(b)$} \\
\hline
$1\underline{2}|3$ & {$\theta_2$} & \motzkinpathnolabel{up,flat-theta,up}{0,0,1}{0,0,0} & 1 & 0 & 0 & 0 & $\{1\}$ \\
\hline
$1\underline{3}|2$ & {$x_3\theta_3$} & \motzkinpathnolabel{up,up,flat-theta}{0,1,1}{0,0,1} & 1 & 0 & 1 & 0 & $\{2\}$ \\
\hline
$21|3$ & {$\xi_2$} & \motzkinpathnolabel{up,flat-xi,up}{0,0,1}{0,0,0} & 0 & 1 & 0 & 0 & $\{1\}$ \\
\hline
$2\underline{3}|1$ & {$x_2x_3\theta_3$} & \motzkinpathnolabel{up,up,flat-theta}{0,1,1}{0,1,1} & 1 & 0 & 2 & 1 & $\{1,2\}$ \\
\hline
$31|2$ & {$x_3\xi_3$} & \motzkinpathnolabel{up,up,flat-xi}{0,1,1}{0,0,1} & 0 & 1 & 1 & 1 & $\{2\}$ \\
\hline
$32|1$ & {$x_2x_3\xi_3$} & \motzkinpathnolabel{up,up,flat-xi}{0,1,1}{0,1,1} & 0 & 1 & 2 & 2 & $\{1,2\}$ \\
\hline
$1|2|3$ & {1} & \motzkinpathnolabel{up,up,up}{0,1,2}{0,0,0} & 0 & 0 & 0 & 0 & $\varnothing$ \\
\hline
$1|3|2$ & {$x_3$} & \motzkinpathnolabel{up,up,up}{0,1,2}{0,0,1} & 0 & 0 & 1 & 1 & $\{2\}$ \\
\hline
$2|1|3$ & {$x_2$} & \motzkinpathnolabel{up,up,up}{0,1,2}{0,1,0} & 0 & 0 & 1 & 1 & $\{1\}$ \\
\hline
$2|3|1$ & {$x_2x_3$} & \motzkinpathnolabel{up,up,up}{0,1,2}{0,1,1} & 0 & 0 & 2 & 2 & $\{1\}$ \\
\hline
$3|1|2$ & {$x_3^2$} & \motzkinpathnolabel{up,up,up}{0,1,2}{0,0,2} & 0 & 0 & 2 & 2 & $\{2\}$ \\
\hline
$3|2|1$ & {$x_2x_3^2$} & \motzkinpathnolabel{up,up,up}{0,1,2}{0,1,2} & 0 & 0 & 3 & 3 & $\{1,2\}$ \\
\hline
\end{tabular}
}
\caption{Conversion between $B_3^{(1,2)}$ and $\SW(1^3)$ with statistics (second half).}
\label{tab:second_half}
\end{table}

\clearpage
\section{Examples of Frobenius and Hilbert series}

\begin{flalign*}\Frob(R_1^{(1,2)};q;u,v) &= s_1 &\end{flalign*}

\begin{flalign*}\Hilb(R_1^{(1,2)};q;u,v) &= 1 &\end{flalign*}

\begin{flalign*}
    \Frob(R_2^{(1,2)};q;u,v) &= (q+u+v)s_{11} + s_2 &
\end{flalign*}

\begin{flalign*}
    \Hilb(R_2^{(1,2)};q;u,v) &= q+u+v+1 &
\end{flalign*}

\begin{flalign*}\Frob(R_3^{(1,2)};q;u,v) &= \Big(q^3+(q^2 +q)u+(q^2+q)v+u^2+(q+1)uv+v^2\Big)s_{1 1 1}\\ 
&\quad + \Big((q^2+q)+(q+1)u+(q+1)v+uv\Big)s_{2 1} + s_{3}&
\end{flalign*}

\begin{flalign*}\Hilb(R_3^{(1,2)};q;u,v) &=
(q^3 + 2q^2 + 2q + 1) + (q^2 + 3q + 2)u + (q^2 + 3q + 2)v + u^2\\&\quad  + (q + 3)uv + v^2     &
\end{flalign*}

\begin{flalign*}\Frob(R_4^{(1,2)};q;u,v) &= \Big(
q^6
+(q^5+q^4+q^3)u
+(q^5+q^4+q^3)v
+(q^3+q^2+q)u^2
\\&\quad\quad
+(q^4+2q^3+2q^2+q)uv
+(q^3+q^2+q)v^2
+(q^2+q+1)u^2v
\\&\quad\quad
+(q^2+q+1)uv^2
+u^3
+v^3\Big)
s_{1 1 1 1} 
\\&\quad+\Big(
(q^5+q^4+q^3)
+(q^4+2q^3+2q^2+q)u
+(q^4+2q^3+2q^2+q)v
\\&\quad\quad
+(q^2+q+1)u^2
+(q^3+3q^2+3q+1)uv
+(q^2+q+1)v^2
\\&\quad\quad
+(q+1)u^2v
+(q+1)uv^2\Big)
s_{2 1 1}
\\&\quad+\Big((q^4+q^2)
+(q^3+q^2+q)u
+(q^3+q^2+q)v
+(q^2+2q+1)uv
\\&\quad\quad
+qu^2+
qv^2+
u^2v+
uv^2\Big)
s_{2 2} 
\\&\quad+\Big((q^3+q^2+q)+  (q^2+q+1)u+ 
(q+1)uv +(q^2+q+1)v\Big)s_{3 1} 
\\&\quad+ s_4&
\end{flalign*}

\begin{flalign*}\Hilb(R_4^{(1,2)};q;u,v) &= 
(q^6 + 3q^5 + 5q^4 +  6q^3 + 5q^2 +3q + 1)\\
&\quad
+(q^5 + 4q^4 + 9q^3 + 11q^2 + 8q +   3)u \\
&\quad + (q^5 + 4q^4 +  9q^3 + 11q^2 + 8q + 3)v
+ (q^3 + 4q^2 + 6q + 3)u^2 \\
& \quad
+ (q^4 +  5q^3 + 13q^2 + 17q +  8)uv
+ (q^3 + 4q^2 + 6q +  3)v^2 \\
&\quad
+ (q^2 + 4q  + 6)u^2v 
+ (q^2 + 4q + 6)uv^2 
 +  u^3 +  v^3
 &
\end{flalign*}

\end{document}